\documentclass[12pt,a4paper,reqno,twoside]{amsart}

\usepackage[english]{babel}
\usepackage{stmaryrd}
\usepackage{dsfont}
\usepackage[symbol*,ragged]{footmisc}
\usepackage[colorlinks,linkcolor=red,anchorcolor=blue,citecolor=blue,urlcolor=blue]{hyperref}
\usepackage{color,xcolor}

\usepackage{geometry}
\usepackage{amssymb}
\usepackage{amsmath}
\usepackage{mathrsfs}
\usepackage{amsfonts}
\usepackage{epsfig}

\usepackage{amsthm}
\usepackage{amsxtra}
\usepackage{bbding}
\usepackage{epsfig}
\usepackage{graphicx}
\usepackage{latexsym}
\usepackage{mathbbol}
\usepackage{bbold}

\usepackage{pifont}
\usepackage{wasysym}
\usepackage{skull}
\usepackage{float}

\DeclareSymbolFontAlphabet{\mathbb}{AMSb}
\DeclareSymbolFontAlphabet{\mathbbol}{bbold}

\usepackage{amscd}
\usepackage[all]{xy}
\allowdisplaybreaks[4]
\usepackage{setspace}

\newcommand{\mr}{\ensuremath{\mathbb R}}

\newcommand{\dif}{\mathrm{d}}

\newcommand{\rear}{\mathbb{R}}

\newcommand{\sumstar}{\sideset{}{^*}\sum}

\newcommand{\mz}{\ensuremath{\mathbb Z}}
\newcommand{\leg}[2]{\left(\frac{#1}{#2}\right)}

\theoremstyle{plain}
\newtheorem{theorem}{\normalfont\scshape Theorem}[section]
\newtheorem{proposition}{\normalfont\scshape Proposition}[section]
\newtheorem{lemma}[proposition]{\normalfont\scshape Lemma}

\newtheorem*{corollary*}{\normalfont\scshape Corollary}
\newtheorem{remark}{\normalfont\scshape Remark}

\theoremstyle{remark}
\newtheorem*{remark*}{\normalfont\scshape Remark}

\numberwithin{equation}{section}
\addtocounter{footnote}{1}

\renewcommand{\footnoterule}{
  \kern -3pt
  \hrule width 2.5in height 0.4pt
  \kern 3pt
}

\makeatletter
\@ifundefined{MakeUppercase}{}{}
\makeatother

\begin{document}
	
\title[Bounds for moments of  twisted quadratic  characters of prime modulus]
	  {Bounds for moments of  twisted quadratic  characters of prime modulus}

\author[P. Gao]{Peng Gao}
\address{School of Mathematical Sciences, Beihang University, Beijing 100191, China}
\email{penggao@buaa.edu.cn}

\author[Y. Zhao]{Yuetong Zhao}
\address{School of Mathematical Sciences, China University of Geosciences (Beijing), Beijing 100191, China}
\email{yuetong.zhao.math@gmail.com}

\date{}

\footnotetext[1]{

\quad\,\,
{\textbf{MR(2020) Subject Classification}}: 11L40, 11M06 
}
	
\begin{abstract}
  We study, under the Generalized Riemann Hypothesis (GRH), the moments of sums of Fourier coefficients of a fixed holomorphic Hecke eigenform twisted by the quadratic character $\chi_{8p}$, where $p$ ranges over odd primes. We establish the correct order of magnitude for the unsmoothed $m$-th moment for all real $m\geq 4$, and a sharp upper bound of order 
$XY^{m/2}(\log X)^{m(m-3)/2}$for the smoothed $m$-th moment for all integers 
$m\geq 4$. A matching lower bound for all even integers 
$m\geq 4$ shows that this bound is optimal.
\end{abstract}
\maketitle

\section{Introduction and Main Result}

Moments of families of 
$L$-functions at the central point are a prominent topic in analytic number theory and have numerous applications. In 1981, M. Jutila \cite{Jutila} initiated the study of the first and second moments of the family of quadratic Dirichlet 
$L$-functions 
$L(1/2,\chi_d)$, where 
$d$ runs over fundamental discriminants and $\chi_d=\leg {d}{\cdot}$
 denotes the Kronecker symbol, obtaining asymptotic formulas for these moments. In the same paper, Jutila \cite{Jutila} also established an asymptotic formula for the first moment of the family of quadratic Dirichlet 
$L$-functions of prime modulus $L(1/2,\chi_p)$, for primes 
$p$ satisfying certain congruence conditions, thus resolving a conjecture posed earlier by D. Goldfeld and C. Viola \cite{goldfeldviola}. Subsequently, S. Baluyot and K. Pratt \cite{B&P} derived, under the Generalized Riemann Hypothesis (GRH), an asymptotic formula for the second moment of $L(1/2,\chi_p)$. Moreover, analogous results in the function field setting for quadratic Dirichlet $L$-functions of prime modulus have been obtained by J. Andrade and J. P. Keating \cite{andradekeating} and by H. M. Bui and A. Florea \cite{B&F}. These developments demonstrate that the study of moments in the prime modulus case is of significant interest and has attracted considerable attention.
\smallskip

Using ideas from random matrix theory, J. P. Keating and N. C. Snaith \cite{Keating-Snaith02} conjectured in 2000 that for any real number $k>0$,
\begin{equation*}
    \sum_{|d|\leq X}L(1/2,\chi_d)^k\sim C_kX(\log X)^{\frac{k(k+1)}{2}},
\end{equation*}
where $C_k$ is an  explicit constant. In 2009, K. Soundararajan \cite{Sound2009} developed a method to obtain conditional upper bounds for such moments that are close to the conjectured order, assuming the GRH. This approach was subsequently refined by A. J. Harper \cite{Harper}, who derived  sharp conditional upper bounds for these moments of $L$-functions. Combining the ideas of Soundararajan \cite{Sound2009} and Harper \cite{Harper}, the first author and L. Zhao \cite{gaozhao2023} investigated  the prime modulus case and, under GRH, established the correct order of magnitude for the $k$-th moment of  $L(1/2,\chi_p)$. In a related direction concerning  shifted moments of $L$-functions at points on the critical line, B. Szab\'o \cite{Szab} obtained sharp upper bounds  and applied them to prove that for any real $k>4$, any sufficiently large integer $q$, and any real $y$ satisfying $2\leq y\leq q^{1/2}$,
\begin{align*}
\sum_{\chi\in X^*_q}\Big|\sum_{n\leq y}\chi(n)\Big|^{k}\ll_k\phi(q)y^{k/2}(\log y)^{(k/2-1)^2},
\end{align*}
where $X^*_q$ denotes the set of primitive Dirichlet characters modulo $q$, and $\phi(q)$ is Euler's totient function. In particular, bounds for quadratic Dirichlet character sums form an important application of the study of moments of quadratic Dirichlet 
$L$-functions.
\smallskip

For $m>0$, we define the $m$-th moment of quadratic character sums by
\begin{align*}
 S_m(X,Y):= \sumstar_{0<d \leq X }\Big | \sum_{n \leq Y}\chi_{8d}(n)\Big |^{m},
\end{align*}
 where $X$ and $Y$ are positive real numbers, and throughout the paper  $\sumstar$ denotes a sum restricted to  odd, square-free integers. A conjecture of Jutila \cite{Jutila} asserts that for any positive integer $m$, there exist constants $c_1(m)$, $c_2(m)>0$, depending only on $m$, such that
 \begin{align*}
 S_m(X,Y)\leq c_1(m)XY^{m/2}(\log X)^{c_2(m)}.
 \end{align*}
 In \cite{G&Zhao2024},  the first author and L. Zhao confirmed a smoothed version of Jutila's conjecture  under the GRH, using upper bounds for moments of quadratic Dirichlet $L$-functions. More precisely, they proved that for large $X$, $Y$ and any real $m\geq 1/2$,
\begin{align*}
 S_m(X,Y;\Psi):=\sumstar_{\substack{0<d \leq X }}\Big | \sum_{n \geq 1}\chi_{8d}(n)\Psi\Big(\frac{n}{Y}\Big)\Big |^{m} \ll XY^{m/2}(\log X)^{\frac {m(m+1)}{2}},
\end{align*}
where $\Psi$ is a smooth weight function.  Furthermore,  in
 \cite[Theorem 1.3]{G&Zhao2024-3}, they showed  that under GRH, for any real $m>\sqrt{5}+1$,
 \begin{align*}
 S_m(X,Y)\ll XY^{m/2}(\log X)^{\frac{m(m-1)}{2}+1},
 \end{align*}
  by employing  estimates for shifted moments of quadratic Dirichlet $L$-functions. Moreover, the second author \cite{zhao2025} obtained an analogous bound when restricting to the family  of prime moduli. Recently,  M. Munsch and Y. Toma \cite{MT25} improved this bound in the smoothed setting for  integers $m$.  Assuming GRH, they proved that for any  integer $m\geq2$,
   \begin{align*}
 S_m(X,Y;\Psi)\ll XY^{m/2}(\log X)^{\frac{m(m-1)}{2}},
 \end{align*}
   thus reducing the  exponent from $\frac {m(m-1)}{2}+1$ to the optimal $\frac {m(m-1)}{2}$ for integer $m \geq 2$. 
   This result is consistent with earlier work of M. V. Armon \cite[Theorem 2]{Armon}, who unconditionally established the optimal bound  for $S_1(X,Y)$. Moreover, Munsch and Toma \cite{MT25}  also obtained unconditional lower bounds for $S_m(X,Y)$
 for all even integer moments over a wide range of parameters.
\smallskip

Let $f$ be a fixed holomorphic Hecke eigenform of weight $\kappa\equiv 0 \pmod{4}$ for the full modular group $SL_2(\mz)$. The Fourier expansion of  $f$ at infinity is given by
\begin{align*}
f(z)=\sum_{n=1}^{\infty}\lambda_f(n)n^{\frac{\kappa-1}{2}}e(nz),
\end{align*}
where $e(z)=e^{2\pi iz}$.
 We consider moments of sums involving the Fourier coefficients $\lambda_f (n)$ twisted by quadratic characters $\chi_{8d}(n)$, where $d$ runs over odd, positive, square-free integers.   Define
\begin{align*}
 T_{m}(X,Y;f):=  \sumstar_{0<d \leq X }\Big | \sum_{n \leq Y}\chi_{8d}(n)\lambda_f(n)\Big |^{m},
\end{align*}
 where $m$, $X$,  $Y$ are positive real numbers.
In \cite{G&Zhao24-12}, the first author and L. Zhao  proved that for any real  $m\geq 4$,
\begin{align}\label{G&Z}
T_m(X,Y;f)\ll XY^{m/2}(\log X)^{\frac{(m-2)(m-1)}{2}+1}.
\end{align}
In a subsequent work, the authors \cite{G&Zyt} obtained an asymptotic formula in the smoothed case when $m=2$. Specifically, 
for smooth, compactly supported non-negative   functions $\Phi, \Psi$ and  for large $X, Y$ satisfying $Y \ll X^{1-\varepsilon}$ for any $\varepsilon>0$,
\begin{align*}
    \sumstar_{d} \Bigg(\sum_{n} \lambda_f(n)\chi_{8d}(n)\Phi \Big(\frac nY \Big)\Bigg)^2 \Psi \Big(\frac d{X} \Big) \sim  XY.
\end{align*}
   Combining this  with \eqref{G&Z}  suggests that one should expect,  for all real $m \geq 2$,
\begin{align}
\label{expect}
 T_{m}(X,Y;f)\ll XY^{m/2}(\log X)^{\frac {m(m-3)}{2}+1}.
\end{align}
In a later work by the authors \cite{G&ZY-25},  it was shown under GRH that this exponent $\frac {m(m-3)}{2}+1$ can be improved to $\frac {m(m-3)}{2}$ for  integer moments in the  smoothed setting. More precisely, for any integer $m\geq 4$,
\begin{align*}
 T_{m}(X,Y;f,\Phi):=\sumstar_{0<d \leq X }\Big | \sum_{n \leq Y}\chi_{8d}(n)\lambda_f(n)\Phi\Big(\frac{n}{Y}\Big)\Big |^{m}\ll XY^{m/2}(\log X)^{\frac {m(m-3)}{2}},
\end{align*}
where  $\Phi$ is a non-negative  smooth function. In the same paper, unconditional  lower bounds were obtained  for all even integers  $m\geq 4$, showing optimality  in those cases.
\smallskip

In this paper, we focus on twisted character sums of the quadratic Dirichlet character $\chi_{8p}$, where 
$p$ runs over the odd prime numbers. For real numbers $m>0$, we define
\begin{align*}
 U_{m}(X,Y;f):=  \sum_{2<p\leq X }(\log p)\Big | \sum_{n \leq Y}\chi_{8p}(n)\lambda_f(n)\Big |^{m}.
\end{align*}
 We first apply the method of the first author and L. Zhao \cite {G&Zhao24-12} to obtain the corresponding result for the prime modulus case associated with \eqref{G&Z}. 

\begin{theorem}\label{real}
With the notation as above and assuming the truth of GRH, for any integer $k\geq 1$ and any real number $m$ satisfying $m\geq 2k+2$, we have for large $Y\le X$ and any $\varepsilon>0$,
\begin{align*}
U_m(X,Y;f)\ll XY^{m/2} (\log X)^{R(m,k,\varepsilon)},
\end{align*}
where
\begin{align}\label{Edef}
  \begin{split}
R(m,k,\varepsilon)=\max\Big(\frac{m^2}{2}-\frac{3m}{2}+k+1, \frac{(m-2k)^2}{4}&+2k^2-2k+1+\varepsilon, \\
&\frac{m(m-1)}{2}+2k(2k-m)+\varepsilon\Big).
  \end{split}
\end{align}
In particular, by setting $k=1$, we obtain, for $m\geq 4$,
\begin{align*}
U_m(X,Y;f)\ll XY^{m/2} (\log X)^{\frac{m(m-3)}{2 }+2}.
\end{align*}
\end{theorem}
\begin{remark}
  From Theorem \ref{real}, we derive that for any real number $m\geq 4$ and $\varepsilon>0$,
\begin{equation}\label{good}
    U_m(X,Y;f)\ll XY^{m/2}(\log X)^{O_{m,\varepsilon}(1)}.
\end{equation}
As for $0<m<4$, we can use H\"older's inequality to get
\begin{align*}
 U_m(X,Y;f) &\ll  X^{1-1/n}(U_{mn}(X,Y;f))^{1/n}\ll X^{1-1/n}(XY^{mn/2}(\log X)^{O_{m,\varepsilon}(1)})^{1/n}\\
 &\ll XY^{m/2}(\log X)^{O_{m,\varepsilon}(1)},
\end{align*}
where the real number $n>4$ is large enough such that $mn> 4$. It follows that for any real number $m>0$, (\ref{good}) holds.
\end{remark}

It should be noted that Theorem \ref{real} holds for all real numbers 
$m$ satisfying the given conditions. Below, we consider the case where 
$m$ is restricted to integers, which allows us to obtain a result sharper than Theorem \ref{real}, saving more powers of 
$\log X$. For  simplicity, we consider the smoothed moments. Define
\begin{align*}
U_m(X,Y;f,W):=\sum_{2<p\leq X}(\log p) \Big|\sum_{n\geq 1}\chi_{8p}(n)\lambda_f(n)W\Big(\frac{n}{Y}\Big)\Big|^m,
 \end{align*}
where $W$ is  smooth, non-negative,  and compactly supported on $(0,\infty)$. By  repeated integration by parts,  the Mellin transform $\widehat{W}$ of $W$ satisfies,  for any integer $E\geq 1$ and $\Re(s)\ge 1/2$,
\begin{align}\label{Phibound}
\widehat{W}(s)\ll \frac{1}{(1+|s|)^{E}}.
\end{align}

We provide an upper bound for $U_m(X,Y;f,W)$.
  
\begin{theorem}\label{upper}
  With the notation as above and assuming the truth of GRH, for any integer $m\geq 4$ and sufficiently large $Y\leq X$, we have
  \begin{equation}
  \label{Tupper}
  U_m(X,Y;f,W) \ll XY^{m/2} (\log X)^{\frac{m(m-3)}{2 }}.
  \end{equation}
\end{theorem}

The proof of Theorem \ref{upper} is motivated by the work of Munsch and Toma \cite{MT25} concerning the moments of quadratic Dirichlet character sums. Our proof also relies on a careful application of bounds for shifted moments of twisted modular 
$L$-functions with prime moduli, which will be stated in Section \ref{shiftedmoments}, the proof of which will be postponed to Section \ref{proof2.13}.
\smallskip

   We also follow the approach in \cite{MT25} to establish  a commensurate lower bound for $U_m(X,Y;f,W)$.
\begin{theorem}\label{lower}
  With the notation as above and assuming the truth of GRH, for any even integer $m\geq 4$ and sufficiently large $X$, we have for $X^{\varepsilon} \ll Y\ll X^{1-\alpha}$ for any $\alpha>0$,
\begin{equation*}
  U_m(X,Y;f,W) \gg XY^{m/2} (\log X)^{\frac{m(m-3)}{2 }}.
\end{equation*}
\end{theorem}

   In view of Theorems \ref{upper} and \ref{lower}, we see that the upper bound for $U_m(X,Y;f,W)$ given in \eqref{Tupper} is indeed optimal for even integers $m \geq 4$.

\section{Preliminaries}\label{sec 2}

\textbf{Notation.}
Let $p$ and $q$ always denote prime numbers. The symbol $X$ denotes a large real number. The symbol $\varepsilon$ denotes an arbitrarily small positive number, which may not be the same in different occurrences. For any real number $t$, we write $\lfloor t\rfloor$ for its integer part. The symbol $\square$ denotes a perfect square. We use $\chi_{c}(n)$ to denote the Kronecker symbol $\big(\frac{c}{n}\big)$. We write $d|n$ to indicate that $d$ divides $n$, and $p^k\parallel n$ to indicate that $p^k$ is the exact power of $p$ dividing $n$. The function $\Omega(n)$ denotes the total number of prime factors of $n$, counted with multiplicity.

\subsection{Cusp form $L$-functions}
\label{sec:cusp form}

    Recall that $f$ is a fixed holomorphic Hecke eigenform of weight $\kappa \equiv 0 \pmod 4$ for the full modular group $SL_2 (\mathbb{Z})$. The associated modular $L$-function $L(s, f)$ for $\Re(s)>1$ is then defined to be
\begin{align}
\label{Lphichi}
L(s, f ) &= \sum_{n=1}^{\infty} \frac{\lambda_f(n)}{n^s}
 = \prod_{p} \left(1 - \frac{\lambda_f (p)}{p^s}  + \frac{1}{p^{2s}}\right)^{-1}=\prod_{p} \left(1 - \frac{\alpha_p }{p^s} \right)^{-1}\left(1 - \frac{\beta_p }{p^s} \right)^{-1}.
\end{align}
 By Deligne's proof \cite{D} of the Weil conjecture, we know that
\begin{align*}
|\alpha_{p}|=|\beta_{p}|=1, \quad \alpha_{p}\beta_{p}=1.
\end{align*}
It follows that $\lambda_f(n) \in \mathbb{R}$ with $\lambda_f(1) = 1$, and
\begin{align}
\label{lambdabound}
\begin{split}
  |\lambda_f(n)| \leq d(n) \ll n^{\varepsilon},
\end{split}
\end{align}
where $d(n)$ denotes the number of positive divisors of $n$.  
The last bound follows from the well-known estimate $d(n) \ll n^{\varepsilon}$ for any $\varepsilon > 0$ (see \cite[Theorem 2.11]{MVa1}).
\smallskip

 The symmetric square $L$-function $L(s, \operatorname{sym}^2 f)$ of $f$ is defined for $\Re(s)>1$ by
 (see \cite[page 137]{iwakow} and \cite[(25.73)]{iwakow})
\begin{align}
\label{Lsymexp}
\begin{split}
 L(s, \operatorname{sym}^2 f)&= \prod_p(1-\alpha^2_pp^{-s})^{-1}(1-p^{-s})^{-1}(1-\beta^2_pp^{-s})^{-1} = \zeta(2s) \sum_{n \geq 1}\frac {\lambda_f(n^2)}{n^s}\\
 &=\prod_{p}\bigg(1-\frac {\lambda_f(p^2)}{p^s}+\frac {\lambda_f(p^2)}{p^{2s}}-\frac {1}{p^{3s}}\bigg )^{-1}.
\end{split}
\end{align}
From \eqref{Lphichi} and \eqref{Lsymexp}, we derive that
\begin{align}\label{sumlambdapsquare}
\alpha_p+\beta_p=\lambda_f(p)\quad \text{and} \quad \alpha^2_p+\beta^2_p=\lambda_f^2(p)-2=\lambda_f(p^2)-1.
\end{align}
  It follows from a result of G. Shimura \cite{Shimura} that $L(s, \operatorname{sym}^2 f)$ has no pole at $s=1$. Moreover, the corresponding completed symmetric square $L$-function
\begin{align}
\label{SymsquareLfeqn}
 \Lambda(s, \operatorname{sym}^2 f)=& \pi^{-3s/2}\Gamma \Big(\frac {s+1}{2}\Big)\Gamma\Big (\frac {s+\kappa-1}{2}\Big) \Gamma\Big (\frac {s+\kappa}{2}\Big) L(s, \operatorname{sym}^2 f)
\end{align}
  is entire and satisfies the functional equation 
$\Lambda(s, \operatorname{sym}^2 f)=\Lambda(1-s, \operatorname{sym}^2 f)$.
\medskip

\subsection{Sum over primes}
This section gathers several asymptotic results concerning sums over primes.

\begin{lemma}
\label{RS}
 Let $x \geq 2$. We have, for some constants $b_1, b_2$,
\begin{align}
\label{merten}
\sum_{p\le x} \frac{1}{p} =& \log \log x + b_1+ O\Big(\frac{1}{\log x}\Big), \\
\label{mertenpartialsummation}
\sum_{p\le x} \frac {\log p}{p} =& \log x + O(1), \quad \mbox{and} \\
\label{merten1}
\sum_{p\le x} \frac{\lambda^2_f(p)}{p} =& \log \log x + b_2+ O\Big(\frac{1}{\log x}\Big).
\end{align}
\end{lemma}
\begin{proof}
  The expressions in \eqref{merten} and \eqref{mertenpartialsummation} can be found in parts (d) and (b) of \cite[Theorem 2.7]{MVa1}, respectively. The estimate given in \eqref{merten1} follows from \cite[Lemma 2.1]{GHH}.  
\end{proof}

\begin{lemma}\label{RS3}
  We have, for $x \geq e^{2e}$ and $\alpha \geq 0$,
\begin{align}
\label{mertenstype}
  \sum_{p\leq x} \frac{\cos(\alpha \log p) }{p}&= \log |\zeta(1+1/\log x+i\alpha)| +O(1)\nonumber\\
  &\leq 
\begin{cases}
\log\log x+O(1)            & \text{if }  \alpha\leq 1/\log x \text{ or } \alpha\geq e^x  ,   \\
\log(1/\alpha)+O(1)        & \text{if }  1/\log x\leq \alpha \leq 10,   \\
\log\log\log \alpha + O(1) & \text{if }   10 \leq \alpha \leq e^x.
\end{cases} \\
\label{mertenstypesympower}
  \sum_{p\leq x} \frac{\cos(\alpha \log p)\lambda_f(p^2) }{p} &=\log |L(1+1/\log x+i\alpha, \operatorname{sym}^2 f)| +O(1)\nonumber\\
  &\leq 
\begin{cases}
 O(1)            & \text{if }  \alpha \leq e^e,   \\
\log\log\log \alpha + O(1)        & \text{if }  e^e \leq \alpha \leq e^x,   \\
 \log\log x & \text{if }  \alpha\geq e^x.
\end{cases}
\end{align}
The estimates in the first two cases of \eqref{mertenstype} are unconditional and the third holds under the Riemann hypothesis. The estimates in the first and last cases of \eqref{mertenstypesympower} hold unconditionally while the second is subject to GRH.
\end{lemma}
\begin{proof}
  The formula in \eqref{mertenstype} is a special case of \cite[Lemma 3.2]{Kou}, while the other bounds follow from \cite[Lemma 2]{Szab}.
To establish \eqref{mertenstypesympower}, we apply \cite[Lemma 3.2]{Kou} together with \eqref{Lsymexp} and \eqref{sumlambdapsquare}, which yield the equality in \eqref{mertenstypesympower}.
For a detailed argument, see \cite[Lemma 2.6]{G&Zhao24-12}.
\end{proof}

\subsection{Mean Value Estimates}

  In this section, we include some mean value estimates that are needed in our proof of the main results in this paper. 
Let $\Phi$ be a non-negative smooth function such that
\begin{equation}\label{Phi}
\Phi(x)\begin{cases}
            \leq 1,\quad  x\in [1/4,1/2]\cup [1,3/2],\\
            =1,\quad  x\in[1/2,1],\\
            =0,\quad \text{otherwise }.
            \end{cases}
\end{equation}
The Mellin transform of $\Phi(x)$ is denoted by $\widehat{\Phi}(s)$, and  for any complex number $s$, it is given by the following integral:
\begin{equation}\label{Mellin}
    \widehat{\Phi}(s)=\int_0^{\infty}\Phi(x)x^{s-1}dx.
\end{equation}
  \begin{lemma}
\label{prsum}
Assume GRH.  Let $c$ be a positive odd integer, $\Phi(x)$ be a smooth function as described above, and let $k\in \rear$ with $k \geq 0$. Then, for any $\varepsilon>0$, we have
\begin{align*}
\sum_{(p,2)=1}  (\log p)\Big(1-\frac{1}{p}\Big)^{-k} & \chi_{8p}(c) \Phi \Big( \frac{p}{X} \Big)
= \delta_{c=\square}{\widehat\Phi}(1) X +
O_k(X^{1/2+\varepsilon}\log \log(c+2)),
\end{align*}
  where $\delta_{c=\square}=1$ if $c$ is a square and $\delta_{c=\square}=0$ otherwise.
\end{lemma}
\begin{proof}
This result can be deduced from \cite[Lemma 2.4]{G&Zhao6}; see also \cite[Lemma 2.8]{zhao2025} for details.
\end{proof}

\begin{lemma}\label{character}
Let $A>0$ be fixed. Assume the GRH.  
Then for all positive integers $n \le X^A$ with $X \ge 2$, we have
\begin{align*}
    \sum_{0<p \leq X} (\log p) \chi_{8p}(n)=X\cdot\mathds{1}_{n=\square}+O(X^{1/2}\log^2 X),
\end{align*}
where the implied constant depends only on $A$, and 
$\mathds{1}_{n=\square}$ denotes the indicator function of perfect squares.
\end{lemma}
\begin{proof}
  If $n$ is a perfect square, then $\chi_{8p}(n) = 1$ for all odd primes $p \nmid n$. Hence
\begin{align*}
  \sum_{0<p \leq X} (\log p) \chi_{8p}(n)=
  \sum_{\substack{0<p \leq X\\ p\nmid n}} (\log p) =\sum_{0<p \leq X} (\log p) +O(\omega(n)),
\end{align*}
where $\omega(n)$ denotes the number of distinct prime divisors of $n$. 
Since $\omega(n) \ll \log n$ and $n \le X^A$, it follows that $\omega(n) \ll_A \log X$. By the prime number theorem under GRH,
\begin{align*}
\sum_{0 < p \le X} (\log p) = X + O(X^{1/2}\log^2 X),
\end{align*}
and therefore
\begin{align*}
\sum_{0<p \leq X} (\log p) \chi_{8p}(n)=X+O(X^{1/2}\log^2 X).
\end{align*}
 If $n$ is not a perfect square, then by \cite[Lemma~2]{A&B} (which holds under GRH),
 \begin{align*}
\sum_{0<p \leq X} (\log p) \chi_{8p}(n)\ll X^{1/2}\log^2 X.
\end{align*}
Combining the two cases gives the desired result.
  \end{proof}

Our next result concerns the analytical properties of certain Dirichlet series.

\begin{lemma}\label{Dirichletseries}
  Define
  \begin{align*}
  \begin{split}
g(n_1, \ldots, n_{m}):= & \prod_{1\leq i\leq m}\lambda_f(n_i)\cdot \mathds{1}_{n_1\cdots n_m=\square}, \\
g^+(n_1, \ldots, n_{m}):= & \prod_{1\leq i\leq m}|\lambda_f(n_i)|\cdot \mathds{1}_{n_1\cdots n_m=\square}.
\end{split}
  \end{align*}
We have
\begin{align}
 \label{Gdef}
  \begin{split}
G({\bf s}):=& \sum_{n_1, \ldots, n_{m}=1}^{\infty} \frac{g(n_1, \ldots, n_{m})}{n_1^{s_1} \cdots n_m^{s_{m}}} \\
= & \left( \prod_{1 \leq j \leq m} L(2s_j, \operatorname{sym}^2f ) \prod_{1 \leq l_1 < l_2 \leq m} \zeta(s_{l_1} + s_{l_2})L (s_{l_1} + s_{l_2},\operatorname{sym}^2f)\right) E(s_1, \ldots, s_{m}), \\
 G^+({\bf s}):=& \sum_{n_1, \ldots, n_{m}=1}^{\infty} \frac{g^+(n_1, \ldots, n_{m})}{n_1^{s_1} \cdots n_m^{s_{m}}}\\
  =&  \left( \prod_{1 \leq j \leq m} L(2s_j, \operatorname{sym}^2f ) \prod_{1 \leq l_1 < l_2 \leq m} \zeta(s_{l_1} + s_{l_2})L (s_{l_1} + s_{l_2},\operatorname{sym}^2f)\right) E^+(s_1, \ldots, s_{m}),
\end{split}
\end{align}
where \( E(s_1, \ldots, s_{m}), E^+(s_1, \ldots, s_{m})  \) are Dirichlet series absolutely convergent for \( \Re(s_j) > 1/4 \) for \( 1 \leq j \leq m \).
\end{lemma}
\begin{proof}
 By adapting the proof of \cite[Lemma 2.4]{G&ZY-25}, we can obtain this result.
\end{proof}

   We now apply Lemma \ref{Dirichletseries} to establish a result on sums of $\prod_{i=1}^{m}\lambda_f(n_i)$ restricted to the set $n_1\cdots n_m=\square$. Let $W$ be a non-negative smooth function. We also denote by $\widehat{W}$ the Mellin transform of $W$ and note that repeated integration by parts implies that for any integer $E\geq 0$ and $\Re(s)\geq 1/2$,
\begin{align*}
\widehat{W}(s)\ll (1+|s|)^{-E}.
\end{align*}
\begin{lemma}
\label{sumoversquare}
  With the notation as above, let $\beta_i >0$ and define for $m \geq 1$,
\begin{equation*}
 P_f(Y^{\beta_1}, \cdots, Y^{\beta_m}):=\sum_{\substack{n_1,\dots,n_m \geq 1\\ n_1\cdots n_m=\square}} \prod_{i=1}^{m}\lambda_f(n_i)W\Big(\frac{n_i}{Y^{\beta_i}}\Big).
\end{equation*}
  Then we have for $m \geq 4$,
\begin{align}
\label{Pfest}
 P_f(Y^{\beta_1}, \cdots, Y^{\beta_m}) \asymp Y^{\sum^m_{i=1}\beta_i/2}(\log Y)^{\frac {m(m-3)}{2}}.
\end{align}
\end{lemma}
\begin{proof}
  We define for $m \geq 1$,
\begin{align*}
 P^+_f(Y^{\beta_1}, \cdots, Y^{\beta_m}):=& \sum_{\substack{n_1,\dots,n_m \geq 1\\ n_1\cdots n_m=\square}}  \prod_{i=1}^{m}|\lambda_f(n_i)|W\Big(\frac{n_i}{Y^{\beta_i}}\Big).
\end{align*}

   We then apply Mellin inversion to the variables $n_1,\dots, n_{m}$ to see that
\begin{align*}
\begin{split}
P_f(Y^{\beta_1}, \cdots, Y^{\beta_m})=& \Big(\frac{1}{2\pi i}\Big)^{m}\int_{(2)}\cdots\int_{(2)}Y^{\sum^m_{i=1}u_i\beta_i}G(u_1,\dots,u_{m})
\widehat{W}(u_1)\cdots\widehat{W}(u_{m})
\dif u_1\cdots \dif u_{m}, \\
P^+_f(Y^{\beta_1}, \cdots, Y^{\beta_m})=& \Big(\frac{1}{2\pi i}\Big)^{m}\int_{(2)}\cdots\int_{(2)}Y^{\sum^m_{i=1}u_i\beta_i}G^+(u_1,\dots,u_{m})
\widehat{W}(u_1)\cdots\widehat{W}(u_{m})
\dif u_1\cdots \dif u_{m},
\end{split}
\end{align*}
where $G, G^+$ are defined in \eqref{Gdef}. Next, applying Lemma \ref{Dirichletseries}, we write the Dirichlet series $G(\mathbf{s})$ and $G^+(\mathbf{s})$ as products of zeta and $L$-functions multiplied by regular factors $E(\mathbf{s})$ and $E^+(\mathbf{s})$, respectively. As these product representations are identical in form to those in \cite[Lemma 2.4]{G&ZY-25}, and the remainder factors $E(\mathbf{s}), E^+(\mathbf{s})$ retain good convergence properties in $\Re(s_j) > 1/4$, the argument follows directly from the proof of \cite[Lemma 2.5]{G&ZY-25}.
\end{proof}

\subsection{Upper bound for $\log |L(1/2+it, f \otimes \chi_{8p})|$ }

\begin{lemma}
\label{logupper}
   Assume GRH for $L(s, f \otimes\chi_{8p})$ for all odd primes $p$. Let $k$ be a positive integer and let $A,a_1,a_2,\ldots, a_{k}$ be fixed positive real constants and $x\geq 2$.  Set $a:=a_1+\cdots+ a_{k}$.  Suppose $X$ is a large real number and $t_1,\ldots, t_{k}$ are fixed real numbers with $|t_i|\leq X^A$. For any integer $n$, let
   \begin{align}\label{defh}
h(n):=\frac{1}{2}  \Re\Big(\sum^{k}_{m=1}a_mn^{-it_m}\Big).
   \end{align}
 Then, we have
\begin{align}
\label{mainupperquad}
\begin{split}
 &\sum^{k}_{m=1} a_m \log |L(1/2+it_m, f \otimes \chi_{8p})| \\
    \leq & 2 \sum_{q\leq x} \frac{h(q)\lambda_f(q)\chi_{8p}(q)}{q^{1/2+ 1/\log x}}\frac{\log (x/q)}{\log x}
    +\sum_{q\leq x^{1/2}} \frac{h(q^2)(\lambda_f(q^2)-1)}{q}\\
    &\quad- 3a\log \Big(1-\frac{1}{p}\Big)+2(A+1)a\frac{\log X}{\log x}+ O\left( \frac {\log X}{x^{1/2}\log x}+1 \right).
\end{split}
\end{align}
\end{lemma}
\begin{proof}
  By \cite[Corollary 2.11]{G&Zhao24-12}, we have
  \begin{align}\label{GZ}
    \begin{split}
& \sum^{k}_{m=1}  a_m\log |L(1/2+it_m,f \otimes \chi_{8p})| \\
    \leq & 2 \sum_{q\leq x} \frac{h(q)\chi_{8p}(q)\lambda_f(q)}{q^{1/2+1/\log x}}\frac{\log x/q}{\log x}
    +\sum_{q\leq x^{1/2}} \frac{h(q^2)\chi_{8p}(q^2)(\lambda_f(q^2)-1)}{q}\\
    &\qquad+2(A+1)a\frac{\log X}{\log x}+O\left( \frac {\log X}{x^{1/2}\log x}+1 \right).
    \end{split}
\end{align}
For the second term on the right-hand side of the above inequality, we obtain
\begin{align}
\label{sump}
 \sum_{q\leq x^{1/2}} \frac{h(q^2)\chi_{8p}(q^2)(\lambda_f(q^2)-1)}{q}=& \sum_{q\leq x^{1/2}} \frac{h(q^2)(\lambda_f(q^2)-1)}{q}
 -\frac{h(p^2)(\lambda_f(p^2)-1)}{p} \nonumber\\
 \leq & \sum_{q\leq x^{1/2}} \frac{h(q^2)(\lambda_f(q^2)-1)}{q}+ \frac{3a}{p}\nonumber\\
 \leq & \sum_{q\leq x^{1/2}} \frac{h(q^2)(\lambda_f(q^2)-1)}{q}- 3a\log \Big(1-\frac{1}{p}\Big).
\end{align}
Here the final step uses the inequality $x\leq-\log(1-x)$ for any $0<x<1$. Inserting  (\ref{sump}) into (\ref{GZ}), we obtain the result.
\end{proof}

\subsection{Shifted moments of quadratic twists of modular $L$-functions }\label{shiftedmoments}
  To prove the upper bound in Theorem \ref{upper}, we present a prime-modular analogue of \cite[Theorem 1.2]{G&Zhao24-12}, which concerns upper bounds on shifted moments of quadratic twists of modular $L$-functions. The proof of the following theorem  relies  crucially on \cite{Szab} and on sharp upper bounds for shifted moments of the corresponding $L$-functions under GRH, obtained via a method of Soundararajan \cite{Sound2009} and its refinement by  Harper \cite{Harper}.  The proof of the following theorem will be given in Section \ref{proof2.13}.
\begin{theorem}\label{sumL}
With the notation as above and assuming the truth of GRH. Let $k\geq 1$ be a fixed integer and $A>0$ a fixed constant. Suppose $X$ is a large real number and ${\bf t}=(t_1,\dots,t_k)$ is a real $k$-tuple with $|t_j|\leq X^A$. Then
\begin{align*}
  &\sum_{2<p \leq X}(\log p) \prod_{1\leq j\leq k}\big| L\big(1/2+it_j,f \otimes \chi_{8p} \big) \big|^{a_j}\\
  \ll &
  X(\log X)^{\frac{1}{4}\sum_{j=1}^k a_j^2}
   \prod_{1\leq j<l\leq k}\Big|\zeta\Big(1+i(t_j-t_l)+\frac{1}{\log X}\Big)\Big|^{a_ja_l/2}\Big|\zeta\Big(1+i(t_j+t_l)+\frac{1}{\log X}\Big)\Big|^{a_ja_l/2} \\
  &\times\prod_{1\leq j\leq k} \Big|\zeta\Big(1+2it_j+\frac{1}{\log X}\Big)\Big|^{a^2_j/4-a_j/2} \prod_{1\leq j<l\leq k}\Big|L\Big(1+i(t_j-t_l)+\frac{1}{\log X}, \operatorname{sym}^2f\Big)\Big|^{a_ja_l/2}\\
  &\times\Big|L\Big(1+i(t_j+t_l)+\frac{1}{\log X}, \operatorname{sym}^2f\Big)\Big|^{a_ja_l/2} \prod_{1\leq j\leq k} \Big|L\Big(1+2it_j+\frac{1}{\log X}, \operatorname{sym}^2f\Big)\Big|^{a^2_j/4+a_j/2}.
\end{align*}
 Here the implied constant depends on $k$, $A$ and the $a_j$'s, but not on $X$ or the $t_j$'s.
  \end{theorem}

  Directly from Theoem \ref{sumL} and Lemma \ref{RS3}, we have the following result.
\begin{proposition}\label{prop}
With the notation as above and assuming the truth of GRH, let $k\geq 1$ be a fixed integer and $A>0$ a fixed constant. Suppose $X$ is a large real number and ${\bf t}=(t_1,\dots,t_k)$ is a real $k$-tuple with $|t_j|\leq X^A$. Then
\begin{align*}
  \sum_{2<p \leq X}(\log p) \prod_{1\leq j\leq k}\big| L\big(1/2+it_j,f \otimes \chi_{8p} \big) \big|^{a_j}
 \ll  X(\log X)^{\frac{1}{4}\sum_{j=1}^k a_j^2}  \mathcal{G}_1\cdot\mathcal{G}_2,
\end{align*}
where
\begin{align}\label{G1G2}
  \begin{split}
\mathcal{G}_1=\prod_{1\leq i<j\leq k} g_1(|t_i-t_j|)^{a_ia_j/2}g_1(|t_i+t_j|)^{a_ia_j/2}\prod_{1\leq i\leq k} g_1(|2t_i|)^{a^2_i/4-a_i/2},\\
\mathcal{G}_2=\prod_{1\leq i<j\leq k} g_2(|t_i-t_j|)^{a_ia_j/2}g_2(|t_i+t_j|)^{a_ia_j/2}\prod_{1\leq i\leq k} g_2(|2t_i|)^{a^2_i/4+a_i/2},
\end{split}
\end{align}
and the functions $g_1, g_2 : \mathbb{R}_{\geq 0} \rightarrow \mathbb{R}$ are defined by
\begin{align}
\label{g1Def}
\begin{split}
g_1(x) =\begin{cases}
\log X,  & \text{if } x\leq 1/\log X \text{ or } x \geq e^X, \\
1/x, & \text{if }   1/\log X \leq x\leq 10, \\
\log \log x, & \text{if }  10 \leq x \leq e^{X},
\end{cases}\,\,
g_2(x) =
\begin{cases}
1,  & \text{if } x \leq e^e, \\
\log \log x, & \text{if }   e^e \leq x \leq e^{X}, \\
 \log X, & \text{if }  x \geq e^{X}.
\end{cases}
\end{split}
\end{align}
 Here the implied constant depends on $k$, $A$ and the $a_j$'s, but not on $X$ or the $t_j$'s.
\end{proposition}

Finally, we derive a somewhat rough upper bound, which can be obtained by modifying the proof of \cite[Theorem 2]{Harper}. 
\begin{lemma}
\label{rude}
Assume  GRH. Let $k\geq 1$ be a fixed integer and ${\bf a}=(a_1,\ldots ,a_{k}),{\bf t}=(t_1,\ldots ,t_{k})$
 be $k$-tuples of real numbers with $a_i \geq 0$ for all $i$.  Then, uniformly for $\sigma \geq 1/2$ and sufficiently large $X$, we have
\begin{align*}
   \sum_{2<p \leq X}(\log p )\prod_{1\leq j\leq k}\big| L\big(\sigma+it_j,f \otimes \chi_{8p} \big) \big|^{a_j}\ll_{{\bf a}} &  X(\log X)^{O(1)}.
\end{align*}
\end{lemma}

\section{Proof of Theorem \ref{sumL}}\label{proof2.13}
 First, we introduce some notation that will be used throughout this section.
Let $M$ and $N$ be large numbers that depend only on $\mathbf{a}$. For $j\geq 1$, define
\begin{align}\label{notations}
\begin{split}
 &\beta_{0} = 0, \;\;\;\;\; \beta_{j} = \frac{20^{j-1}}{(\log\log X)^{2}}, \;\;\; s_j=2\lfloor e^N\beta_j^{-3/4}\rfloor, \;\;\;\;\; h_{j} = \frac{s_j}{100}, \\
 &\mathcal{J} = 1 + \max\{j : \beta_{j} \leq 10^{-M} \}.
\end{split}
\end{align}
From these definitions, we have
\begin{equation}\label{sizeJ}
    \beta_{\mathcal{J}-1}\leq 10^{-M}\Rightarrow 20^{\mathcal{J}-2}\leq10^{-M}(\log\log X)^2\Rightarrow \mathcal{J}\ll_{\bf a} \log \log\log X.
\end{equation}
Let $\gamma(n)$ denote the multiplicative function such that $\gamma(p^\alpha)=\alpha !$. We see that
\begin{equation}\label{gamma}
    \gamma(n^2)\geq \gamma (n).
\end{equation}
Recall the definition of $h(n)$ in \eqref{defh}. For any $1\leq i\leq j\leq \mathcal{J}$, we define
\begin{equation}\label{aG}
    a{\mathcal G}_{i,j}(p)=\sum_{X^{\beta_{i-1}}<q\leq X^{\beta_i}}\frac{2h(q)\lambda_f(q)\chi_{8p}(q)}{q^{1/2+1/\log X^{\beta_j}}}\frac{\log(X^{\beta_j}/q)}{\log X^{\beta_j}}
\end{equation}
and
\begin{equation}\label{t(n,x)}
t\big(n,X^{\beta_j}\big)=\prod_{q^{\alpha}||n}\left(\frac{2h(q)\lambda_f(q)\log(X^{\beta_j}/q)}{aq^{1/\log X^{\beta_j}}\log X^{\beta_j}}\right)^{\alpha}.
\end{equation}
We observe that $t(n,x)$ is a totally multiplicative function of $n$ and
\begin{equation}\label{sizet}
|t(n,x)|\leq 1.
\end{equation}

We define the weight function $\Phi$ as in \eqref{Phi}. By partitioning the interval $0<p \leq X$ into dyadic blocks, we see that in order to 
 prove Theorem \ref{sumL}, it suffices to show that
\begin{align}
\label{sum2.15}
\begin{split}
  &\sum_{ (p,2)=1}  (\log p)\prod_{1\leq j\leq k}\big| L\big(1/2+it_j,f \otimes \chi_{8p} \big) \big|^{a_j}\Phi \Big( \frac p{X} \Big) \\
  \ll &
  X(\log X)^{\frac{1}{4}\sum_{j=1}^k a_j^2}  \prod_{1\leq j<l \leq k} \big|\zeta (1+i(t_j-t_l)+\tfrac 1{\log X} ) \big|^{a_ja_l/2}\big|\zeta(1+i(t_j+t_l)+\tfrac 1{\log X}) \big|^{a_ja_l/2}\\
  &\times\prod_{1\leq j\leq k} \big|\zeta(1+2it_j+\tfrac 1{\log X}) \big|^{a^2_j/4-a_j/2}  \prod_{1\leq j<l \leq k} \Big|L \Big(1+i(t_j-t_l)+\tfrac 1{\log X}, \operatorname{sym}^2 f \Big) \Big|^{a_ja_l/2}\\
  &\times \Big|L \Big(1+i(t_j+t_l)+\tfrac 1{\log X}, \operatorname{sym}^2 f \Big) \Big|^{a_ja_l/2} 
\prod_{1\leq j\leq k} \Big|L \Big(1+2it_j+\tfrac 1{\log X}, \operatorname{sym}^2 f \Big) \Big|^{a^2_j/4+a_j/2}.
\end{split}
\end{align}
Taking $x=X^{\beta_j}$ in (\ref{mainupperquad})  yields 
\begin{align*}
\begin{split}
 & \sum^{k}_{m=1}a_m\log |L(1/2+it_m,f\otimes\chi_{8p})| \\
 \le& \,a\,\sum^{j}_{l=1} {\mathcal G}_{l,j}(p)+\sum_{q\leq X^{\beta_j/2}}
 \frac{h(q^2)(\lambda_f(q^2)-1)}{q}- 3a\log \Big(1-\frac{1}{p}\Big)+2(A+1)a\beta^{-1}_j+O(1).
\end{split}
\end{align*}
We now decompose the main Dirichlet polynomial into smaller pieces $a{\mathcal G}_{i,j}(p)\,(1\leq i\leq j\leq \mathcal{J})$. In terms of the size of the short Dirichlet polynomial, we define the following sets:
\begin{align}\label{sets}
  \begin{split}
  \mathcal{S}(0) =& \{ (p,2)=1 : |a{\mathcal G}_{1,l}(p)| > h_{1} \; \text{ for some } 1 \leq l \leq \mathcal{J} \} ,   \\
 \mathcal{S}(j) =& \{ (p,2)=1  : |a{\mathcal G}_{m,l}(p)| \leq
  h_{m},  \; \mbox{for all} \; 1 \leq m \leq j \; \mbox{and} \; m \leq l \leq \mathcal{J}, \\
 & \;\;\;\;\;\;\;\;\;\;\; \text{but }  |a{\mathcal G}_{j+1,l}(p)| > h_{j+1} \; \text{ for some } j+1 \leq l \leq \mathcal{J} \} ,  \quad\quad  1\leq j \leq \mathcal{J}, \\
 \mathcal{S}(\mathcal{J}) =& \{(p,2)=1  : |a{\mathcal G}_{m,
\mathcal{J}}(p)| \leq h_{m}, \; \forall 1 \leq m \leq \mathcal{J}\}.
  \end{split}
\end{align}
Clearly, 
\begin{equation*}
    \{p:(p,2)=1\}=\bigcup_{j=0}^{\mathcal{J}}\mathcal{S}(j).
\end{equation*}
We now estimate the sum (\ref{sum2.15}) over the sets $\mathcal{S}(0)$ and $\mathcal{S}(j),(1\leq j\leq \mathcal{J})$, respectively.

\subsection{The estimate of the sum over the set $\mathcal{S}(0)$}

First, we state the following lemma. The proof is almost the same as that of \cite[Lemma 3.1]{zhao2025}, so we omit it.
\begin{lemma}\label{sumPhi}
Let $\Phi$ and $\mathcal{S}(0)$ be defined as in (\ref{Phi}) and (\ref{sets}). Then, we have
\begin{equation*}
\sum_{\substack{(p,2)=1 \\p \in \mathcal{S}(0) }} (\log p)\Phi \Big( \frac p{X} \Big)\ll_{\bf a} Xe^{-(\log\log X)^2}.
\end{equation*}
\end{lemma}

Using the Cauchy-Schwarz inequality, we have
\begin{align*}
  & \sum_{\substack{(p,2)=1 \\p \in S(0) }}  (\log p)\prod_{1\leq j\leq k}\big| L\big(1/2+it_j,f \otimes \chi_{8p} \big) \big|^{a_j}\Phi \Big( \frac p{X} \Big) \nonumber\\
\leq &  \Big ( \sum_{\substack{(p,2)=1 \\p \in S(0) }} (\log p)\Phi \Big( \frac p{X} \Big)  \Big )^{1/2} \Big (
 \sum_{\substack{(p,2)=1 }}(\log p)\prod_{1\leq j\leq k}\big| L\big(1/2+it_j,f \otimes \chi_{8p} \big) \big|^{2a_j}\Phi \Big( \frac p{X} \Big)\Big)^{1/2}.
\end{align*}
Then, by Lemma \ref{rude}  and Lemma \ref{sumPhi}, the above summation is 
\begin{align}\label{strongerbound}
\ll_{\bf a} X(\log X)^{a(a+1)/4}e^{-\frac{1}{2}(\log\log X)^2}.
\end{align} 
Moreover, by \cite[Theorem 6.7]{MVa1}, for $|t_m|,|t_{\ell}|\leq X^A$ and $X$ large enough, we have
\begin{align*}
\Big|\zeta(1+i(t_m-t_{\ell})+\frac 1{\log X})\Big|&\gg \min\Big(\frac 1{\log X}, \frac 1{\log(|t_m-t_{\ell}|+4)}\Big)\gg \frac1{\log X}.
\end{align*}
On the other hand, from \cite[Theorem 1.3]{goldfeldli2018}, we can get
\begin{align*}
  \Big|L(1+i(t_m-t_{\ell})+\frac 1{\log X}, \operatorname{sym}^2 f)\Big|&\gg \frac1{(\log X)^{3/2}}.
\end{align*}
Hence, the upper bound 
\begin{align*}
&\sum_{\substack{(p,2)=1 \\p \in \mathcal{S}(0) }}  (\log p)\prod_{1\leq j\leq k}\big| L\big(1/2+it_j,f \otimes \chi_{8p} \big) \big|^{a_j}\Phi \Big( \frac p{X} \Big) \\
    \ll &
  X(\log X)^{\frac{1}{4}\sum_{j=1}^k a_j^2}  \prod_{1\leq j<l \leq k} \big|\zeta (1+i(t_j-t_l)+\tfrac 1{\log X} ) \big|^{a_ja_l/2}\big|\zeta(1+i(t_j+t_l)+\tfrac 1{\log X}) \big|^{a_ja_l/2}\\
  &\times\prod_{1\leq j\leq k} \big|\zeta(1+2it_j+\tfrac 1{\log X}) \big|^{a^2_j/4-a_j/2}  \prod_{1\leq j<l \leq k} \Big|L \Big(1+i(t_j-t_l)+\tfrac 1{\log X}, \operatorname{sym}^2 f \Big) \Big|^{a_ja_l/2}\\
  &\times \Big|L \Big(1+i(t_j+t_l)+\tfrac 1{\log X}, \operatorname{sym}^2 f \Big) \Big|^{a_ja_l/2} 
\prod_{1\leq j\leq k} \Big|L \Big(1+2it_j+\tfrac 1{\log X}, \operatorname{sym}^2 f \Big) \Big|^{a^2_j/4+a_j/2}
\end{align*}
follows immediately, as it is implied by the stronger bound \eqref{strongerbound}.

\subsection{The estimate of the sum over the set $\mathcal{S}(j)$}

Fix a $j$ with $1 \leq j \leq \mathcal{J}$. From Lemma \ref{logupper}, we have
\begin{align*}
 &\sum_{\substack{(p,2)=1 \\p \in \mathcal{S}(j)}}(\log p)\prod_{1\leq \ell\leq k}\big| L\big(1/2+it_{\ell},f \otimes \chi_{8p} \big) \big|^{a_{\ell}}\Phi\Big(\frac{p}{X}\Big) \nonumber\\
\ll & \sum_{\substack{(p,2)=1 \\p \in \mathcal{S}(j)}}\exp \Big(
 a\sum^j_{i=1}{\mathcal G}_{i,j}(p)\Big )\exp \left(\sum_{q\leq X^{\beta_j/2}} \frac{h(q^2)(\lambda_f(q^2)-1)}{q} \right) \nonumber\\
 &\quad\quad\quad \times\exp \left(\frac {2(Q+1)a}{\beta_j} \right)(\log p)\left(1-\frac{1}{p}\right)^{-3a}\Phi \Big(\frac{p}{X}\Big).
\end{align*}

Next, we follow the work of \cite{zhao2025} (pages 1567--1572) to obtain
\begin{align}\label{Sj4}
  \begin{split}
    &\sum_{\substack{(p,2)=1 \\p \in \mathcal{S}(j) }}(\log p)\prod_{1\leq \ell \leq k}\big| L\big(1/2+it_{\ell},f \otimes \chi_{8p} \big) \big|^{a_{\ell}}\Phi\Big(\frac {p}{X}\Big)\\
 \ll & X \exp \Big (-\frac {(Q+1)a}{\beta_j} + \sum_{q \leq X}\frac {2h^2(q)\lambda_f^2(q)}{q}+\sum_{q\leq X} \frac{h(q^2)(\lambda_f(q^2)-1)}{q} \Big )\\
 =& X \exp \Big (-\frac {(Q+1)a}{\beta_j} \Big) \exp\Big(\sum_{q \leq X} \lambda_f(q^2) \big(\frac {2h^2(q)}{q}+\frac{h(q^2)}{q}\big)\Big)\exp\Big(\sum_{q\leq X} \big(\frac{2h^2(q)}{q}-\frac{h(q^2)}{q} \big)\Big )
\end{split}
\end{align}

   Using the trigonometric identities
   \begin{align*}
    \cos(x+y)+\cos(x-y)=2\cos(x)\cos(y)\quad \text{and}\quad \cos^2(x)=\frac{1}{2}(\cos(2x)+1),
   \end{align*}
   we get
\begin{align}\label{hexp}
& \sum_{q \leq X} \lambda_f(q^2) \big(\frac {2h^2(q)}{q}+\frac{h(q^2)}{q}\big)= \sum_{q \leq X }\frac {\lambda_f(q^2)}{2q}\Big ( \Big( \sum^{k}_{m=1}a_m\cos(t_m\log q) \Big)^2+\sum^{k}_{m=1}a_m\cos(2t_m\log q)\Big ) \nonumber\\
=& \sum_{q \leq X }\frac {\lambda_f(q^2)}{2q}\Big (\sum^{k}_{m=1}a^2_m\cos^2(t_m\log q)+2\sum_{1 \leq m<\ell \leq k}a_ma_{\ell}\cos(t_m\log q)\cos(t_{\ell}\log q)+\sum^{k}_{m=1}a_m\cos(2t_m\log q)\Big )\nonumber \\
=& \frac{1}{4}\sum^{k}_{m=1}a^2_m\sum_{q \leq X }\frac {\lambda_f(q^2)}{q}+\frac{1}{2}\sum_{1 \leq m<\ell \leq k}a_ma_{\ell}\sum_{q\leq X}\lambda_f(q^2)\Big(\frac{\cos((t_m+t_{\ell})\log q)}{q}+\frac{\cos((t_m-t_{\ell})\log q)}{q}\Big )\nonumber\\
&+\sum^{k}_{m=1} \Big( \frac {a^2_m}{4}+\frac{a_{m}}{2} \Big) \sum_{q\leq X}\frac{\cos(2t_{m}\log q)\lambda_f(q^2)}{q}.
\end{align}
Applying (\ref{mertenstypesympower}) to evaluate the right-hand side expression of (\ref{hexp}), we obtain
\begin{align}\label{hexp1}
     &\sum_{q \leq X} \lambda_f(q^2) \big(\frac {2h^2(q)}{q}+\frac{h(q^2)}{q}\big)
     \ll \frac{1}{4}\Big(\sum^{k}_{m=1}a^2_m\Big)\log\log X\nonumber\\&+\frac{1}{2}\sum_{1 \leq m<\ell \leq k}a_ma_{\ell}\log \Big(\Big|L(1+\frac{1}{\log X}+i(t_m+t_{\ell}),\operatorname{sym}^2 f)\Big|\Big|L(1+\frac{1}{\log X}+i(t_m-t_{\ell}),\operatorname{sym}^2 f)\Big|\Big)\nonumber\\
    & +\sum^{k}_{m=1} \Big( \frac {a^2_m}{4}+\frac{a_{m}}{2} \Big) \log|L(1+1/\log X+i2t_{m},\operatorname{sym}^2 f)|.
\end{align}

By the same method and \eqref{mertenstype}, we obtain
\begin{align}\label{hexp2}
     &\sum_{q \leq X }\big(\frac {2h^2(q)}{q}- \frac{h(q^2)}{q}\big)
     \ll \frac{1}{4}\Big(\sum^{k}_{m=1}a^2_m\Big)\log\log X \nonumber\\
     &+\frac{1}{2}\sum_{1 \leq m<\ell \leq k}a_ma_{\ell}\log \Big(\Big|\zeta(1+i(t_m+t_{\ell})+\frac{1}{\log X})\Big|\Big|\zeta(1+i(t_m-t_{\ell})+\frac{1}{\log X})\Big|\Big)\nonumber\\
    & \quad +\sum^{k}_{m=1} \Big( \frac {a^2_m}{4}-\frac{a_{m}}{2} \Big) \log\Big|\zeta(1+i2t_{m}+\frac{1}{\log X})\Big|.
\end{align}
Inserting (\ref{hexp1}) and (\ref{hexp2}) into (\ref{Sj4}), we have
\begin{align}\label{Sj6}
  \begin{split}
&\sum_{\substack{(p,2)=1 \\p \in \mathcal{S}(j) }}(\log p)\prod_{1\leq \ell\leq k}\big| L\big(1/2+it_{\ell},f \otimes \chi_{8p} \big) \big|^{a_{\ell}}\Phi\Big(\frac {p}{X}\Big) \ll \exp \Big (-\frac {(Q+1)a}{\beta_j} \Big) \\
&\times X (\log X)^{\frac{1}{4}\sum_{j=1}^k a_j^2}  \prod_{1\leq j<l \leq k} \big|\zeta (1+i(t_j-t_l)+\tfrac 1{\log X} ) \big|^{a_ja_l/2}\big|\zeta(1+i(t_j+t_l)+\tfrac 1{\log X}) \big|^{a_ja_l/2}\\
  &\times\prod_{1\leq j\leq k} \big|\zeta(1+2it_j+\tfrac 1{\log X}) \big|^{a^2_j/4-a_j/2}  \prod_{1\leq j<l \leq k} \Big|L \Big(1+i(t_j-t_l)+\tfrac 1{\log X}, \operatorname{sym}^2 f \Big) \Big|^{a_ja_l/2}\\
  &\times \Big|L \Big(1+i(t_j+t_l)+\tfrac 1{\log X}, \operatorname{sym}^2 f \Big) \Big|^{a_ja_l/2} 
\prod_{1\leq j\leq k} \Big|L \Big(1+2it_j+\tfrac 1{\log X}, \operatorname{sym}^2 f \Big) \Big|^{a^2_j/4+a_j/2}.
  \end{split}
\end{align}
Since
\begin{equation*}
    \sum_{j=1}^{\mathcal{J}}e^{-\frac {(Q+1)a}{\beta_j}}\ll 1,
\end{equation*}
we derive
\begin{align*}
&\sum_{j=1}^{\mathcal{J}}  \sum_{\substack{(p,2)=1 \\p \in \mathcal{S}(j) }}(\log p)\prod_{1\leq \ell\leq k}\big| L\big(1/2+it_{\ell},f \otimes \chi_{8p} \big) \big|^{a_{\ell}}\Phi\Big(\frac {p}{X}\Big)  \nonumber\\
   \ll &X (\log X)^{\frac{1}{4}\sum_{j=1}^k a_j^2}  \prod_{1\leq j<l \leq k} \big|\zeta (1+i(t_j-t_l)+\tfrac 1{\log X} ) \big|^{a_ja_l/2}\big|\zeta(1+i(t_j+t_l)+\tfrac 1{\log X}) \big|^{a_ja_l/2}\\
  &\times\prod_{1\leq j\leq k} \big|\zeta(1+2it_j+\tfrac 1{\log X}) \big|^{a^2_j/4-a_j/2}  \prod_{1\leq j<l \leq k} \Big|L \Big(1+i(t_j-t_l)+\tfrac 1{\log X}, \operatorname{sym}^2 f \Big) \Big|^{a_ja_l/2}\\
  &\times \Big|L \Big(1+i(t_j+t_l)+\tfrac 1{\log X}, \operatorname{sym}^2 f \Big) \Big|^{a_ja_l/2} 
\prod_{1\leq j\leq k} \Big|L \Big(1+2it_j+\tfrac 1{\log X}, \operatorname{sym}^2 f \Big) \Big|^{a^2_j/4+a_j/2},
\end{align*}
which completes the proof of Theorem \ref{sumL}.

\section{Proof of Theorem \ref{real}}






Define a non-negative smooth function
\begin{equation}\label{PhiD}
\Phi_D(t)\begin{cases}
            \leq 1,\quad  x\in [0,1/D]\cup [1-1/D,1],\\
            =1,\quad  x\in[1/D,1-1/D],\\
            =0,\quad \text{otherwise },
            \end{cases}
\end{equation}
satisfying $\Phi^{(j)}_D(t) \ll_j D^j$ for all integers $j \geq 0$, where $D$ is a parameter, to be specified later. By the Mellin transform and repeated integration by parts, then for any integer $A \geq 1$ and $\Re(s) \geq 1/2$, we have
\begin{align}
\label{whatbound}
 \widehat{\Phi}_D(s)=\int_0^\infty \Phi_D(t)t^{s-1}dt  \ll  D^{A-1}(1+|s|)^{-A}.
\end{align}

Note that 
the convexity inequality for the $m$-th power 
 implies that,  for any $x,y\in\mathbb{C}$ and integer $k\geq 1$, 
$$|x+y|^{k}\leq 2^{k-1}(|x|^{k}+|y|^{k}).$$
Using this inequality and inserting the smooth function $\Phi_D(\frac nY)$ into the definition of $U_m(X,Y;f)$, we obtain
\begin{align}\label{SmXY}
U_m(X,Y;f) \ll U_m^{(1)}(X,Y;f)+U_m^{(2)}(X,Y;f),
\end{align} 	
where
\begin{equation*}
   U_m^{(1)}(X,Y;f):= \sum_{2<p \leq X }(\log p)\Big | \sum_{n}\chi_{8p}(n)\lambda_f(n)\Phi_D \Big( \frac {n}{Y} \Big)\Big |^{m}
\end{equation*}
and
\begin{equation*}
    U_m^{(2)}(X,Y;f):=\sum_{2<p \leq X }(\log p)\bigg|\sum_{n\leq Y} \chi_{8p}(n)\lambda_f(n)\Big(1-\Phi_D \Big( \frac {n}{Y} \Big) \Big)\bigg|^{m}.
\end{equation*}

\begin{proposition}\label{Sm1}
With the notation as above,  and assuming the GRH, we have that for any integer $k \geq 1$ and any real numbers $m \geq 2k+2$ and  $\varepsilon>0$,
    \begin{align*}
    U_m^{(1)}(X,Y;f)\ll XY^{m/2}(\log X)^{ R(m,k,\varepsilon)},
    \end{align*}
 where $R(m,k,\varepsilon)$ is defined in \eqref{Edef}.
\end{proposition}

\begin{proposition}\label{Sm2}
    With the notation as above, and assuming the GRH, we have that for $m \geq 4$,
\begin{equation*}
U_m^{(2)}(X,Y;f) \ll XY^{m/2}.
\end{equation*}
\end{proposition}
By (\ref{SmXY}), Proposition \ref{Sm1} and Proposition \ref{Sm2},  we complete the proof of Theorem \ref{real}. 
\smallskip

In what follows, we present the proofs of Proposition \ref{Sm1} and Proposition \ref{Sm2}.

\subsection{Proof of Proposition \ref{Sm1}}

To handle $U_m^{(1)}(X,Y;f)$, we  follow the initial manipulations performed in Section 5.1 of \cite{zhao2025} with parameter $D=X^{\varepsilon}$ and arrive at the bound
\begin{align}
\label{first}
U^{(1)}_m(X,Y;f)
   \ll  Y^{m/2} \sum_{2<p \leq X }(\log p)
   \Big | \int\limits_{\substack{ |t| \leq X^{\varepsilon}}}\Big |L( 1/2+it, f\otimes\chi_{8p})\Big |\frac 1{1+|t|}  dt\Big |^{m}
   +O(XY^{m/2}).
\end{align}

Proposition \ref{Sm1} follows from the following lemma.

\begin{lemma}\label{fdiff}
With the notation as above,  and assuming the truth of GRH, we have that for any integer $k \geq 1$ and any real numbers $m \geq 2k+2$ and  $\varepsilon>0$,
\begin{equation*}
\sum_{2<p \leq X }(\log p)
   \bigg | \int\limits_{\substack{ |t| \leq X^{\varepsilon}}}\Big |L(1/2+it, f\otimes\chi_{8p})\Big |\frac 1{1+|t|}dt\bigg |^{m} \ll X(\log X)^{ R(m,k,\varepsilon)},
\end{equation*}
where $R(m,k,\varepsilon)$ is defined in \eqref{Edef}.
\end{lemma}
\begin{proof}
By symmetry and H\"older's inequality for $a=1-1/m+\varepsilon$ with $\varepsilon>0$, we get
\begin{align}\label{lemma51}
 &\sum_{2<p \leq X }(\log p)\Big | \int_{ |t| \leq X^{\varepsilon}}   \Big |L(1/2+it, f\otimes\chi_{8p})\Big|\frac {d t}{|t|+1} \Big |^{m} \nonumber\\
 \ll & \sum_{2<p \leq X }(\log p)\Big |\int_0^{X^{\varepsilon}} \frac{|L(1/2+it, f\otimes\chi_{8p})|}{t+1} d t\Big |^{m} \nonumber\\
 \ll & \sum_{2<p \leq X }(\log p)\Big |\sum_{n\leq \log X+1}\int_{e^{n-1}-1}^{e^n-1} \frac{|L(1/2+it, f\otimes\chi_{8p})|}{t+1} d t\Big |^{m}\nonumber\\
  \leq & \sum_{2<p \leq X }(\log p)\bigg(\sum_{n\leq \log X+1} n^{-am/(m-1)} \bigg)^{m-1}
    \sum_{n\leq  \log X+1} \bigg(n^a\int_{e^{n-1}-1}^{e^{n}-1 } \frac{|L(1/2+it, f\otimes\chi_{8p}) |}{t+1} d t\bigg)^{m}   \nonumber\\
   \ll & \sum_{n\leq  \log X+1} \frac{n^{m-1+\varepsilon} }{e^{nm} } \sum_{2<p \leq X }(\log p)\bigg( \int_{e^{n-1}-1}^{e^{n}-1 } |L(1/2+it, f\otimes\chi_{8p}) | d t \bigg)^{m}.
\end{align}
Below we estimate the sum
\begin{equation*}
    T_m(B,X):=\sum_{2<p \leq X }(\log p)\bigg( \int_{0}^{B} |L(1/2+it, f\otimes\chi_{8p}) | d t \bigg)^{m},
\end{equation*}
where $10 \leq B=X^{O(1)}$. The method for obtaining the upper bound for $T_m(B,X)$ is similar to the proof of (6.1) in B. Szab\'o \cite{Szab}.

First, we extract $2k$ integrals, which is permissible since $m\geq 2k+2$ by assumption. Using the notation $d {\mathbf t} =dt_1\cdots dt_k$, we obtain
\begin{align*}
   T_m(B,X)
      \ll \sum_{2<p \leq X }(\log p)\int_{[0,B]^k}\prod_{a=1}^k|L(1/2+ it_a, f\otimes\chi_{8p})|^2 \bigg(\int_{\mathcal{D} }|L(1/2+iu,f\otimes\chi_{8p})| d u \bigg)^{m-2k} d\mathbf{t},
\end{align*}
where $\mathcal{D}=\mathcal{D}(t_1,\ldots,t_k)=\{ u\in [0,B]:|t_1-u|\leq |t_2-u|\leq \cdots \leq |t_k-u| \}$. Here, we restrict the integration over $\mathcal{D}$ by symmetry.


In \cite{zhao2025}, the estimation strategy for $T_m(B,X)$ is introduced in detail, so we omit the details here. 
We define the following dyadic intervals:
\begin{align*}
    &\mathcal{B}_1=\big[-\frac{1}{\log X},\frac{1}{\log X}\big],  \\
    &\mathcal{B}_h=\big[-\frac{e^{h-1}}{\log X}, -\frac{e^{h-2}}{\log X}\big]
  \cup \big[\frac{e^{h-2}}{\log X}, \frac{e^{h-1}}{\log X}\big], \quad \text{for}\quad 2\leq h< \lfloor \log \log X\rfloor+10 =: H,\\
&\mathcal{B}_H=[-B,B]\setminus \bigcup_{1\leq h<H} \mathcal{B}_h.
\end{align*}
Define 
$$\mathcal{A}_h:=\mathcal{B}_h\cap (-t_1+\mathcal{D}),$$
where $c+[a,b]$ denotes the translated interval $[c+a,c+b]$.
For $\mathbf{t}=(t_1,\ldots,t_k)$, we write
\begin{equation}\label{Lnew}
    L(\mathbf{t},u)=\sum_{\substack{p \leq X \\ (p,2)=1}}(\log p)\prod_{a=1}^k|L( 1/2+it_a, f\otimes\chi_{8p})|^2 \cdot |L( 1/2+iu, f\otimes\chi_{8p})|^{m-2k}.
\end{equation}

Through a process similar to that in \cite{zhao2025} (pages 1577--1578), we can obtain
\begin{align}\label{Lintest}
    T_m(B,X)
     \ll  \sum_{1\leq h_0, h_1, \ldots h_{k-1}\leq H} h_0^{m-2k} |\mathcal{B}_{h_0}|^{m-2k-1} \int_{\mathcal{C}_{h_0,h_1, \cdots, h_{k-1}}} L(\mathbf{t},u) \,d u \,d \mathbf{t},
\end{align}
where
$$\mathcal{C}_{h_0,h_1, \cdots, h_{k-1}}=\{(t_1,\ldots,t_k,u)\in [0,B]^{k+1}: u\in t_1+ \mathcal{A}_{h_0},\, |t_{i+1}-u|-|t_i-u|\in \mathcal{B}_{h_i}, \,1 \leq i \leq k-1\}.$$

Using Proposition \ref{prop} to bound $L(\mathbf{t},u)$, for $(t_1,\ldots,t_k,u)\in \mathcal{C}_{h_0,h_1, \cdots, h_{k-1}}$, we have
\begin{align}\label{L(t,u)}
L(\mathbf{t},u)
\ll X(\log X)^{\frac{4k+(m-2k)^2}{4}} \mathcal{G}_1\cdot\mathcal{G}_2,
\end{align}
where $\mathcal{G}_1$ and  $\mathcal{G}_2$ are defined in \eqref{G1G2}.
Below, we provide upper bounds of $g_1(\alpha)$  and $g_2(\alpha)$ for two cases, with $\alpha$ in different ranges,  and deduce the corresponding upper bound for (\ref{Lintest}).

\textbf{Case 1:} $h_0<H$.

By the definition of $\mathcal{C}_{h_0,h_1, \cdots, h_{k-1}}$ and the observation that $t_i, u \geq 0, 1\leq i \leq k$, we have
\begin{equation*}
     \frac{e^{h_0}}{\log X}\ll |t_1-u|\ll |t_1+u|\ll B =X^{O(1)}.
\end{equation*}
 Recalling the definition of $g_1$ in \eqref{g1Def}, we derive
 \begin{equation*}
     g_1(|t_1\pm u|)\ll \frac{\log X}{e^{h_0}} \log \log B.
 \end{equation*}
From the definition of $\mathcal{A}_j$, we know that $|t_2-u|\geq |t_1-u|$, furthermore
\begin{equation*}
    \frac{e^{h_0}}{\log X}+ \frac{e^{h_1}}{\log X}\ll |t_2-u|= |t_1-u|+(|t_2-u|-|t_1-u|)\ll |t_2+u|\ll B=X^{O(1)},
\end{equation*}
which implies that
\begin{equation*}
    g_1(|t_2\pm u|)\ll \frac{\log X}{e^{\max(h_0,h_1) }} \log \log B.
\end{equation*}
Similarly, for any $1 \leq i \leq k$, we have
\begin{equation}\label{gtiu}
    g_1(|t_i\pm u|)\ll \frac{\log X}{e^{\max(h_0,h_1,\ldots, h_{i-1})} }\log \log B.
\end{equation}
Moreover, for any $1 \leq i < j \leq k$, we have
\begin{equation*}
    \sum^{j-1}_{s=i}(|t_{s+1}-u|-|t_s-u|) \leq |t_j-t_i| \leq |t_j+t_i|,
\end{equation*}
then
\begin{equation}\label{gtjti}
    g_1(|t_{j}\pm t_i|)\ll \frac{\log X }{e^{\max(h_i,\ldots, h_{j-1} ) }} \log \log B.
\end{equation}
Trivially, from the definition of $g_1$ and $g_2$ in (\ref{g1Def}), we get
\begin{equation}\label{g2ug2ti}
  \begin{split}
   g_1(|2u|)\ll \log X,\quad &g_1(|2t_i|)\ll \log X,\,\,\,\,\text{for any}\,\,\,1\leq i \leq k,\\
   g_2(x)\ll &\log\log B, \,\,\,\,\text{for any}\,\,\, x\in [e^e, e^X].
  \end{split}
\end{equation}

By (\ref{L(t,u)})--(\ref{g2ug2ti}), we deduce that
\begin{align}\label{L(t,u)1}
     & L(\mathbf{t},u) \nonumber\\
     \ll & X(\log X)^{\frac{(m-2k)^2+4k}{4}+\frac{(m-2k)^2}{4}-\frac{m-2k}{2}}(\log \log B)^{O(1)}\nonumber\\
     &\quad\quad\quad\quad\quad\quad\quad\times\bigg(\prod^{k-1}_{i=0}\frac{\log X}{e^{ \max(h_0,h_1,\ldots, h_{i}) }} \bigg)^{2(m-2k)}
     \bigg(\prod^{k-1}_{i=1} \prod^{k}_{j=i+1}\frac{\log X}{e^{\max(h_i,\ldots, h_{j-1} ) } } \bigg)^4 \nonumber\\
     = & X(\log X)^{\frac{m(m-1)}{2}}(\log \log B)^{O(1)} \\ \nonumber
     &\quad\quad\quad\times\exp\Big( -2(m-2k)\sum^{k-1}_{i=0}\max(h_0,h_1,\ldots, h_{i})-4\sum^{k-1}_{i=1} \sum^{k}_{j=i+1}\max(h_i,\ldots, h_{j-1} )\Big).
\end{align}
The volume of the region $\mathcal{C}_{h_0,h_1, \cdots, h_{k-1}}$ is $\ll  B^k e^{h_0+h_1+\cdots+h_{k-1}} (\log X)^{-k}$, and $ |\mathcal{B}_{h_0}|\ll e^{h_0}/\log X$. Inserting (\ref{L(t,u)1}) into (\ref{Lintest}), we obtain
\begin{align}\label{case1}
       &  \sum_{\substack{1\leq h_0<H\\ 1\leq h_1, \ldots h_{k-1}\leq H}}  h_0^{m-2k} |\mathcal{B}_{h_0}|^{m-2k-1} \int\limits_{\mathcal{C}_{h_0,h_1, \cdots, h_{k-1}}} L(\mathbf{t},u) \,d u \,d\mathbf{t} \nonumber\\
    \ll & X(\log X)^{m^2/2-3m/2+k+1}B^k(\log \log B)^{O(1)}\sum_{\substack{1\leq h_0<H\\ 1\leq h_1, \ldots h_{k-1}\leq H}}  h_0^{m-2k} \times \exp\Big( (m-2k-1)h_0 \nonumber\\
    &+\sum^{k-1}_{i=0}h_i-2(m-2k)\sum^{k-1}_{i=0}\max(h_0,h_1,\ldots, h_{i})-4\sum^{k-1}_{i=1} \sum^{k}_{j=i+1}\max(h_i,\ldots, h_{j-1} )\Big) \nonumber\\
    = & X(\log X)^{m^2/2-3m/2+k+1}B^k(\log \log B)^{O(1)}  \sum_{\substack{1\leq h_0<H\\ 1\leq h_1, \ldots h_{k-1}\leq H}}  h_0^{m-2k}\times
 \exp\Big(-(m-2k)h_0
 \nonumber\\
      & -3\sum^{k-1}_{i=1}h_i -2(m-2k)\sum^{k-1}_{i=1}\max(h_0,h_1,\ldots, h_{i})-\sum^{k-1}_{i=1} \sum^{k}_{j=i+2}\max(h_i,\ldots, h_{j-1} )\Big) \nonumber\\
    \ll &   X(\log X)^{m^2/2-3m/2+k+1}B^k(\log \log B)^{O(1)}.
\end{align}

\textbf{Case 2:} $h_0=H$.

(1) $|u|> 5$.
For this case, we have $e^{h_0}\ll \log X$. Hence, for any $1 \leq i \leq k$, we get
\begin{equation}\label{gtiu1}
    g_1(|t_i\pm u|)\ll \log\log  B.
\end{equation}
The estimate for $g_1(|t_j\pm t_i|)$ is the same as in (\ref{gtjti}).  We estimate $g_1(|2t_i|)$ and $g_2$ by the trivial bound, and since $|u| > 5$, we have
\begin{equation}\label{g2ug2ti1}
  \begin{split}
   g_1(|2u|) \ll \log\log  B,&\quad\quad g_1(|2t_i|) \ll \log X,\,\,\text{for any}\,\,1\leq i\leq k,\\
   g_2(x)\ll &\log\log B, \,\,\,\,\text{for any}\,\,\, x\in [e^e, e^X].
  \end{split}
\end{equation}
By (\ref{L(t,u)}), (\ref{gtjti}), (\ref{gtiu1}) and (\ref{g2ug2ti1}), we derive
\begin{align}\label{L(t,u)2}
     &L(\mathbf{t},u) \nonumber\\
     \ll & X(\log X)^{\frac{(m-2k)^2+4k}{4}}(\log \log B)^{O(1)}
     \bigg(\prod^{k-1}_{i=1} \prod^{k}_{j=i+1}\frac{\log X}{e^{\max(h_i,\ldots, h_{j-1} ) } } \bigg)^4 \nonumber\\
     = & X(\log X)^{\frac{(m-2k)^2}{4}+2k^2-k}(\log \log B)^{O(1)} 
\exp\Big( -4\sum^{k-1}_{i=1} \sum^{k}_{j=i+1}\max(h_i,\ldots, h_{j-1} )\Big).
\end{align}
The volume of the region $\mathcal{C}_{H,h_1, \cdots, h_{k-1}}$ is
\begin{equation}\label{Vol}
    \ll B^{k+1} e^{h_1+\cdots+h_{k-1}} (\log X)^{-k+1}.
\end{equation}
As $|\mathcal{B}_H|\ll B$, we deduce from (\ref{Lintest}), (\ref{L(t,u)2}) and (\ref{Vol}) that
\begin{align}\label{case21}
     &\sum_{\substack{1\leq h_1, \ldots h_{k-1}\leq H}}  H^{m-2k} |\mathcal{B}_{H}|^{m-2k-1} \int\limits_{\substack{\mathcal{C}_{H,h_1, \cdots, h_{k-1}} \\ |u| \geq 5}} L(\mathbf{t},u) \,d u\, d \mathbf{t}  \nonumber\\
     \ll & X(\log X)^{\frac{(m-2k)^2}{4}+2k^2-2k+1}B^{m-k}(\log\log X)^{m-2k} (\log \log B)^{O(1)}\nonumber\\
     & \hspace*{2cm} \times \sum_{\substack{1\leq h_1, \ldots h_{k-1}\leq H}} \exp\Big( \sum^{k-1}_{i=1}h_i-4\sum^{k-1}_{i=1} \sum^{k}_{j=i+1}\max(h_i,\ldots, h_{j-1} )\Big) \nonumber\\
     \ll &   X(\log X)^{\frac{(m-2k)^2}{4}+2k^2-2k+1}B^{m-k}(\log\log X)^{m-2k} (\log \log B)^{O(1)}.
\end{align}

(2) $|u|\leq 5$.
We estimate $g_1(|2u|)$ and $g_2$ by the trivial bound, and since $e^{10}<|t_1-u| \leq |t_i-u| \leq |t_i|+|u|$, $|u|\leq 5$, for any $1 \leq i \leq k$, we have $|t_i|\geq 5$ and
\begin{equation}\label{g2ug2ti2}
  \begin{split}
   g_1(|2u|) \ll &\log X,\quad\quad g_1(|2t_i|) \ll \log \log B,\\
   g_2(x)\ll &\log\log B, \,\,\,\,\text{for any}\,\,\, x\in [e^e, e^X].
  \end{split}
\end{equation}
By (\ref{L(t,u)}), (\ref{gtjti}), (\ref{gtiu1}) and (\ref{g2ug2ti2}), we get
\begin{align}\label{L(t,u)3}
 &L(\mathbf{t},u) \nonumber\\
 \ll & X(\log X)^{\frac{(m-2k)^2+4k}{4}+\frac{(m-2k)^2}{4}-\frac{m-2k}{2}}(\log \log B)^{O(1)}
     \bigg(\prod^{k-1}_{i=1} \prod^{k}_{j=i+1}\frac{\log X}{e^{\max(h_i,\ldots, h_{j-1} ) } } \bigg)^4 \nonumber\\
     = & X(\log X)^{\frac{m(m-1)}{2}+2k(2k-m)}(\log \log B)^{O(1)}  \exp\Big( -4\sum^{k-1}_{i=1} \sum^{k}_{j=i+1}\max(h_i,\ldots, h_{j-1} )\Big).
\end{align}

   Since $|t_i| \geq 5$ for any $1 \leq i \leq k$ when $|u| \leq 5$, the volume of the region $\mathcal{C}_{H,h_1, \cdots, h_{k-1}}$ is $\ll B^{k}$. It follows that
\begin{align}\label{case22}
 &\sum_{\substack{1\leq h_1, \ldots h_{k-1}\leq H}}  H^{m-2k} |\mathcal{B}_{H}|^{m-2k-1} \int\limits_{\substack{\mathcal{C}_{H,h_1, \cdots, h_{k-1}} \\ |u| \leq 5}} L(\mathbf{t},u)\, d u\, d \mathbf{t}  \nonumber\\
     \ll & X(\log X)^{\frac{m(m-1)}{2}+2k(2k-m)}B^{m-k-1}(\log\log X)^{m-2k}(\log \log B)^{O(1)} \nonumber\\
     & \hspace*{4cm} \times \sum_{\substack{1\leq h_1, \ldots h_{k-1}\leq H}} \exp\Big( -4\sum^{k-1}_{i=1} \sum^{k}_{j=i+1}\max(h_i,\ldots, h_{j-1} )\Big) \nonumber\\
     \ll &   X(\log X)^{\frac{m(m-1)}{2}+2k(2k-m)}B^{m-k-1}(\log\log X)^{m-2k}(\log \log B)^{O(1)}.
\end{align}
From (\ref{Lintest}), (\ref{case1}), (\ref{case21}) and (\ref{case22}), we obtain
\begin{align}\label{lemma50}
    &T_m(B,X)\nonumber\\
    \ll& X\Bigg(
    (\log X)^{\frac{m^2}{2}-\frac{3m}{2}+k+1}B^k+(\log X)^{\frac{(m-2k)^2}{4}+2k^2-2k+1}B^{m-k}(\log\log X)^{m-2k}\nonumber\\
    &\quad\quad+(\log X)^{\frac{m(m-1)}{2}+2k(2k-m)}B^{m-k-1}(\log\log X)^{m-2k}\Bigg) (\log \log B)^{O(1)}.
\end{align}
We apply (\ref{lemma50}) with $B=e^n$ to estimate (\ref{lemma51}), for any integer $k \geq 1$ and any real numbers $m \geq 2k+2$ and $\varepsilon>0$. Then, we obtain
\begin{align*}
  &\sum_{n\leq  \log X+1}  \frac{n^{m-1+\varepsilon} }{e^{nm} }\sum_{\substack{p \leq X \\ (p,2)=1}} \log p\bigg( \int_{e^{n-1}-1}^{e^{n}-1 } |L(1/2+it,f\otimes\chi_{8p}) | dt \bigg)^{m}
    \\
     \ll & X\sum_{n\leq  \log X+1} \frac{n^{m-1+\varepsilon} }{e^{nm} } \\
    & \times
     \Big(
    (\log X)^{\frac{m^2}{2}-\frac{3m}{2}+k+1}e^{kn}(\log n)^{O(1)}+(\log X)^{\frac{(m-2k)^2}{4}+2k^2-2k+1}e^{mn-kn}(\log\log X)^{m-2k} (\log n)^{O(1)}\nonumber\\
    &\quad\quad+(\log X)^{\frac{m(m-1)}{2}+2k(2k-m)}e^{(m-k-1)n}(\log\log X)^{m-2k}(\log n)^{O(1)}\Big)\nonumber\\
    &\ll X\Big(
    (\log X)^{\frac{m^2}{2}-\frac{3m}{2}+k+1}
    +(\log X)^{\frac{(m-2k)^2}{4}+2k^2-2k+1+\varepsilon}+(\log X)^{\frac{m(m-1)}{2}+2k(2k-m)+\varepsilon}\Big)\\
    &\ll X(\log X)^{R(m,k,\varepsilon)}.
\end{align*}
We complete the proof of Lemma \ref{fdiff}.
\end{proof}

\subsection{Proof of Proposition \ref{Sm2}}

We apply the Cauchy-Schwarz inequality to obtain
\begin{align}\label{pocs1}
  \begin{split}
   U_m^{(2)}(X,Y;f)
    \leq  \bigg(\sum_{2<p \leq X }&(\log p)\bigg|\sum_{n\leq Y} \chi_{8p}(n)\lambda_f(n)\Big (1-\Phi_D \Big( \frac {n}{Y} \Big) \Big )\bigg|^{2}\bigg)^{\frac{1}{2}}\\
    &\times\bigg(\sum_{2<p \leq X }(\log p)\bigg|\sum_{n\leq Y} \chi_{8p}(n)\lambda_f(n)\Big (1-\Phi_D \Big( \frac {n}{Y} \Big)\Big )\bigg|^{2m-2}\bigg)^{\frac{1}{2}}.
  \end{split}
\end{align}
Applying  \cite[Corollary 2]{DRHB}, 
together with $Y \leq X$ and \eqref{lambdabound}, to the first term of product (\ref{pocs1}), it yields that
\begin{align*}
  &\sum_{2<p \leq X }(\log p)\bigg|\sum_{n\leq Y} \chi_{8p}(n)\lambda_f(n)\Big (1-\Phi_D \Big( \frac {n}{Y} \Big) \Big )\bigg|^{2} \\
  \ll &\log X \sum_{\substack{d \leq X \\ (d,2)=1}}\bigg|\sum_{n\leq Y} \chi_{8d}(n)\lambda_f(n) \Big (1-\Phi_D \Big( \frac {n}{Y} \Big)\Big ) \bigg|^{2}\\
  \ll & (XY)^{\varepsilon}(X+Y)\sum_{\substack{n_1, n_2 \leq Y \\ n_1n_2=\square} }\lambda_f(n_1)\lambda_f(n_2)\Big (1-\Phi_D \Big( \frac {n_1}{Y} \Big)\Big )\Big (1-\Phi_D \Big(\frac {n_2}{Y} \Big) \Big ) \\
  \ll &
  X^{1+\varepsilon}\sum_{\substack{ n_1, n_2 \in[1,Y/D]\cup [Y(1-1/D),Y] \\ n_1n_2=\square} }1.
\end{align*}
  We write $n_1=d m^2_1, n_2=d m^2_2$ with $d$ square-free.  The above is
\begin{align}\label{pocs2}
  \ll& X^{1+\varepsilon}\sum_{\substack{d \leq Y}}\sum_{\substack{ m_1, m_2 \in[1,(Y/(dD))^{1/2}]\cup[(Y(1-1/D)/d)^{1/2}, (Y/d)^{1/2}]}}1 \nonumber\\
  \ll & X^{1+\varepsilon}\sum_{\substack{d \leq Y}} \Big(Y/(dD)+\Big((Y/d)^{1/2}-(Y(1-1/D)/d)^{1/2}\Big)^2 \Big)
\nonumber\\
   \ll & X^{1+\varepsilon}YD^{-1}\ll X^{1-2\varepsilon}Y
\end{align}
with $D=X^{3\varepsilon}$.
For the second term of the product (\ref{pocs1}), we have the upper bound
\begin{align}\label{pocs3}
  \begin{split}
&\sum_{2<p \leq X }(\log p)\bigg|\sum_{n\leq Y} \chi_{8p}(n)\lambda_f(n)\Big (1-\Phi_D \Big( \frac {n}{Y} \Big)\Big )\bigg|^{2m-2}
\\
\ll &\sum_{2<p \leq X }(\log p)\bigg|\sum_{n\leq Y} \chi_{8p}(n)\lambda_f(n)\bigg|^{2m-2}+\sum_{2<p \leq X }(\log p)\bigg|\sum_{n\leq Y} \chi_{8p}(n)\lambda_f(n)\Phi_D \Big( \frac {n}{Y} \Big)\bigg|^{2m-2}.
  \end{split}
\end{align}
We apply Proposition \ref{Sm1} to estimate the second term on the right-hand side of (\ref{pocs3}), yielding
\begin{align}\label{pocs4}
 U^{(1)}_{2m-2}(X,Y;f)
 \ll  XY^{m-1}(\log X)^{O_{m,\varepsilon}(1)}.
\end{align}
 To estimate the first term on the right-hand side of \eqref{pocs3}, we apply Perron's formula \cite[Theorem 1 of Chapter 5 ]{Ka}. Then
\begin{align}\label{Pint}
  \begin{split}
    &\sum_{n\leq Y}\chi_{8p}(n)\lambda_f(n)= \frac 1{ 2\pi i}\int_{1+1/\log Y-iY}^{1+1/\log Y+iY}L(s,f\otimes\chi_{8p}) \frac{Y^s}{s} d s +O(\log Y)\\
     = & \left(\frac 1{ 2\pi i}\int_{1/2-iY}^{1/2+iY} +\frac 1{ 2\pi i}\int_{1+1/\log Y -iY}^{1/2-iY}+\frac 1{ 2\pi i}\int_{1/2+iY}^{1+1/\log Y+iY} \right)L(s,f\otimes\chi_{8p})\frac{Y^s}{s}d s +O(\log Y).
  \end{split}
\end{align}
 Inserting (\ref{Pint}) back into the first term on the right-hand side of \eqref{pocs3}, we get
 \begin{align}\label{pocs5}
     \sum_{2<p \leq X }(\log p)\bigg|\sum_{n\leq Y} \chi_{8p}(n)\lambda_f(n)\bigg|^{2m-2}
     \ll (S_1+S_2+S_3+X(\log X)^{2m-2})\log X,
 \end{align}
where
\begin{align*}
    S_1=\sum_{2<p \leq X}\bigg| \int_{1/2-iY}^{1/2+iY} L(s,f\otimes\chi_{8p})\frac{Y^s}{s} ds\bigg|^{2m-2},\quad
    S_2=\sum_{2<p \leq X}\bigg| \int_{1+1/\log Y-iY}^{1/2-iY} L(s,f\otimes\chi_{8p})\frac{Y^s}{s} ds\bigg|^{2m-2}
\end{align*}
and
\begin{align*}
    S_3=\sum_{2<p \leq X}\bigg| \int_{1/2+iY}^{1+1/\log Y+iY} L(s,f\otimes\chi_{8p})\frac{Y^s}{s} ds\bigg|^{2m-2}.
\end{align*}
The estimates for $S_2$ and $S_3$ are identical, we estimate only $S_3$.
We assume that $Y\geq 10$, otherwise, the lemma is trivial.
Applying H\"older's inequality, we obtain
\begin{align}\label{S3}
 S_3\ll & \sum_{2<p \leq X}\bigg( \int_{1/2+iY}^{1+1/\log Y+iY} |L(s,f\otimes\chi_{8p})| |ds| \bigg)^{2m-2} \nonumber\\
 \ll &  \sum_{2<p \leq X}\int_{1/2+iY}^{1+1/\log Y+iY} |L(s,f\otimes\chi_{8p})|^{2m-2} |ds| \ll X(\log X)^{O(1)}.
\end{align}
For the last bound in (\ref{S3}), we apply Lemma \ref{rude} with $ 1/2 \leq \sigma \leq 1+1/\log Y$ under GRH.
  Next, we bound $S_1$ using H\"older's inequality, Lemma \ref{fdiff} with $k=1$, and the assumption $Y \leq X$.  Thus
\begin{align}\label{S1}
  S_1
  \ll & Y^{m-1}\sum_{2<p \leq X} \bigg( \int_{0}^Y \frac{|L( \tfrac{1}{2}+it,f\otimes\chi_{8p}) |}{t+1} d t \bigg)^{2m-2}  \nonumber\\
  \ll & Y^{m-1}\sum_{2<p \leq X } \bigg( \big(\sum_{n\leq \log Y+2}1\big)^{\frac{2m-3}{2m-2}}\bigg(\sum_{n\leq \log Y+2}\bigg(\int_{e^{n-1}-1}^{e^{n}-1 } \frac{|L( \tfrac{1}{2}+it,f\otimes\chi_{8p}) |}{t+1} d t\bigg)^{2m-2}\bigg)^{\frac{1}{2m-2}}\bigg)^{2m-2}\nonumber\\
   \ll &  Y^{m-1}\sum_{n\leq \log Y+2} \frac{n^{2m-2} }{e^{(2m-2)n}} \sum_{\substack{p \leq X \\ (p,2)=1}} \bigg( \int_{e^{n-1}-1}^{e^{n}-1 } |L( \tfrac{1}{2}+it,f\otimes\chi_{8p}) | d t \bigg)^{2m-2} \nonumber\\
  \ll & Y^{m-1}X(\log X)^{O_{m,\varepsilon}(1)}\Big ( \sum_{n\leq \log Y+2}\frac{n^{2m-2}}{e^{(2m-2)n} }e^n +\sum_{n\leq \log Y+2}n^{2m-2} \Big )\nonumber\\
  \ll & Y^{m-1}X(\log X)^{O_{m,\varepsilon}(1)}.
\end{align}
From \eqref{pocs5}--\eqref{S1}, we deduce that
\begin{equation}
\label{pocs6}
   \sum_{2<p \leq X }(\log p)\bigg|\sum_{n\leq Y} \chi_{8p}(n)\lambda_f(n)\bigg|^{2m-2}\ll XY^{m-1}(\log X)^{O_{m,\varepsilon}(1)}.
\end{equation}
Finally, by \eqref{pocs1}--\eqref{pocs4}, \eqref{pocs6} and $ Y \leq X$, we complete the proof of Proposition \ref{Sm2}.

\section{Proof of Theorem \ref{upper}}

First, applying the same method as in the derivation of \eqref{first}, 
we obtain under GRH that
\begin{align*}
 U_m (X,Y;f,W) \ll &
   Y^{m/2}\sum_{2<p\leq X}(\log p) \Bigg| \int\limits_{ |t|\leq X^{2\varepsilon}}\big|L( 1/2+it, f \otimes\chi_{8p})\big|\frac{1}{(1+|t|)^{10}} \dif t\Bigg|^{m}+O(XY^{m/2}).
\end{align*} 	
   Hence, it remains to show that under GRH,  for any integer $m \geq 4$,
\begin{equation*}
 \sum_{2<p \leq X }(\log p) 
   \Bigg| \int\limits_{ |t|\leq X^{\varepsilon}}\big|L( 1/2+it, f \otimes\chi_{8p})\big|\frac{1}{(1+|t|)^{10}} \dif t\Bigg|^{m} \ll X(\log X)^{\frac{m(m-3)}{2}}.
\end{equation*}
In Lemma \ref{fdiff}, we estimate the same expression. Since here we only consider  integer $m\geq 4$, we can save on the power of $\log X$.

  In view of \eqref{lemma51}, 
  it suffices to prove a refinement of  Lemma \ref{fdiff}
  concerning the integral
\begin{equation*}
    Z_m(B,X):=\sum_{2<p \leq X }(\log p) \bigg( \int_{0}^{B} |L(1/2+it,f\otimes\chi_{8p}) | \dif t \bigg)^{m},
\end{equation*}
where $10 \leq B=X^{O(1)}$.

   More precisely, under GRH, we need to show that  for any integer $m \geq 4$,
\begin{align}
\label{Tbound}
    Z_m(B,X) \ll B^2(\log \log B)^{O_m(1)}X(\log X)^{\frac{m(m-3)}{2}}.
\end{align}

   In what follows, we establish \eqref{Tbound}. Our approach is motivated by the proof of \cite[Proposition 23]{MT25}.
We  first observe that 
\begin{align*}
 Z_m(B,X) 
 =& \sum_{2<p \leq X }(\log p) \bigg( \bigg( \int_{0}^{1/(2\log X)}+ \int_{1/(2\log X)}^{5}+ \int_{5}^{B}\bigg) |L(1/2+it,f\otimes\chi_{8p}) |\dif t \bigg)^{m} \\
 \ll   &\sum^3_{k=1}Z_{m,k}(B,X),
\end{align*}
  where
\begin{align*}
 Z_{m,1}(B,X)=& \sum_{2<p \leq X }(\log p) \bigg( \int_{0}^{1/(2\log X)}|L(1/2+it,f\otimes\chi_{8p}) |\dif t\bigg)^{m}, \\
 Z_{m,2}(B,X)=& \sum_{2<p \leq X }(\log p) \bigg( \int_{1/(2\log X)}^{5}|L(1/2+it,f\otimes\chi_{8p}) |\dif t\bigg)^{m}, \\
 Z_{m,3}(B,X)=& \sum_{2<p\leq X }(\log p) \bigg(\int_{5}^{B}|L(1/2+it,f\otimes\chi_{8p}) |\dif t \bigg)^{m}.
\end{align*}
 By symmetry,  for $1 \leq k \leq 3$, we have
\begin{equation*}
 Z_{m,k}(B,X)
      \ll  \sum_{2<p \leq X}(\log p) \int\limits_{\mathcal{A}_{B,k}}\prod_{a=1}^m|L(1/2+ it_a, f\otimes\chi_{8p})| \dif \mathbf{t},
\end{equation*}
where $\mathcal{A}_{B,k}=\{ (t_1,\dots,t_m) \in I_k^m: 0 \leq t_1 \leq t_2 \dots \leq t_m\}$. Here we define $I_1=[0,1/(2\log X)], I_2=[1/(2\log X),5], I_3=[5, B]$.

It follows from Proposition \ref{prop} that, for $1 \leq k \leq 3$, 
\begin{align}
\label{TBXbound}
\begin{split}
 Z_{m,k}(B,X)\ll& X(\log X)^{m/4} \int\limits_{\mathcal{A}_{B,k}} \prod_{1\leq i<j\leq m} g_1(|t_i-t_j|)^{1/2}g_1(|t_i+t_j|)^{1/2}g_2(|t_i-t_j|)^{1/2}g_2(|t_i+t_j|)^{1/2}\\
 &\qquad\qquad\qquad\qquad\times \prod_{1\leq i\leq m} g_1(|2t_i|)^{-1/4}g_2(|2t_i|)^{3/4}  \dif \mathbf{t}  \\
\ll & X(\log X)^{m/4}(\log \log B)^{O_m(1)} \\
& \qquad \times \int\limits_{\mathcal{A}_{B,k}} \prod_{1\leq i<j\leq m} g_1(|t_i-t_j|)^{1/2}g_1(|t_i+t_j|)^{1/2} \prod_{1\leq i\leq m} g_1(|2t_i|)^{-1/4}  \dif \mathbf{t} \\
:= & X(\log \log B)^{O_m(1)}I_{m,B,k},
\end{split}
\end{align}
  where the second bound follows by trivially estimating $g_2$  using \eqref{g1Def}, and 
\begin{align*}
I_{m,B,k} = &(\log X)^{m/4}  \int\limits_{\mathcal{A}_{B,k}} \prod_{1\leq i<j\leq m} g_1(|t_i-t_j|)^{1/2}g_1(|t_i+t_j|)^{1/2} \prod_{1\leq i\leq m} g_1(|2t_i|)^{-1/4}\dif \mathbf{t}.
\end{align*}
The estimate for $I_{m,B,k}$ was obtained in the course of the proof of  \cite[Theorem 1.1]{G&ZY-25}. Specifically, it was shown in \cite{G&ZY-25} that, for all integers $m\geq 4$ and $1 \leq k \leq 3$,
\begin{equation}
\label{ind_hyp}
I_{m,B,k} \ll B^2 (\log X)^{\frac{m(m-3)}{2}}(\log \log B)^{O_m(1)}.
\end{equation}

   The desired bound in \eqref{Tbound} then follows directly from \eqref{TBXbound} and \eqref{ind_hyp}. This completes
   the proof of Theorem \ref{upper}.

\section{Proof of Theorem \ref{lower}}

  For any $Y\geq 1$, and odd prime $2<p\leq X$, we define
$$H_{p}(Y;f,W):= \sum_{n \geq 1} \chi_{8p}(n)\lambda_f(n)W\Big(\frac{n}{Y}\Big).$$
   We further define
\begin{equation*}
\mathcal{H}_1= \sum_{2<p\leq X} (\log p) H_{p}(Y;f,W) (H_{p}(Y^{\varepsilon};f,W))^{m-1}, \quad \quad
\mathcal{H}_2=\sum_{2<p\leq X}(\log p) \vert H_{p}(Y^{\varepsilon};f,W) \vert^m.
\end{equation*}
Applying H\"older's inequality, we obtain
\begin{align}
\label{Holder}
 \mathcal{H}_1^{m} \leq
 \mathcal{H}_2^{m-1} U_m(X,Y;f,W).
 \end{align}

 We evaluate $ \mathcal{H}_2$ using the facts that $\lambda_f(n) \in \mr$ and that $m$ is even. Then, by Lemma \ref{character}, we have
  \begin{align}
\label{eqS2}
\begin{split}
 \mathcal{H}_2 = &\sum_{2<p\leq X}(\log p) ( H_{p}(Y^{\varepsilon};f,W) )^m \\
=& \sum_{2<p\leq X} (\log p) \,\sum_{n_1,\dots,n_m \geq 1} \chi_{8p}(n_1\cdots n_m)\lambda_f(n_1)\cdots \lambda_f(n_m)W\Big(\frac{n_1}{Y^{\varepsilon}}\Big)\cdots W\Big(\frac{n_m}{Y^{\varepsilon}}\Big) \\
  = & X\sum_{\substack{n_1,\dots,n_m \geq 1\\ n_1\cdots n_m=\square}} \lambda_f(n_1)\cdots \lambda_f(n_m)W\Big(\frac{n_1}{Y^{\varepsilon}}\Big)\cdots W\Big(\frac{n_m}{Y^{\varepsilon}}\Big) +O(X^{1/2}\log ^2X\sum_{n_1,\dots,n_m \ll Y^{\varepsilon}}|\lambda_f(n_1)\cdots \lambda_f(n_m)|) \\
= & X\sum_{\substack{n_1,\dots,n_m \geq 1\\ n_1\cdots n_m=\square}} \lambda_f(n_1)\cdots \lambda_f(n_m)W\Big(\frac{n_1}{Y^{\varepsilon}}\Big)\cdots W\Big(\frac{n_m}{Y^{\varepsilon}}\Big)+O(X^{1/2}\log ^2X\sum_{n_1,\dots,n_m \ll Y^{\varepsilon}}(n_1\dots n_m)^{\varepsilon}),
\end{split}
\end{align}
where the last estimation follows from \eqref{lambdabound}.
  Note that
\begin{align}\label{S21}
  X^{1/2}\log ^2X \sum_{n_1,\dots,n_m \ll Y^{\varepsilon}}(n_1\dots n_m)^{\varepsilon} \ll X^{1/2}Y^{(1+\varepsilon)m \varepsilon}\log ^2X =X^{1/2+\varepsilon}.
\end{align}
  Moreover, by Lemma \ref{sumoversquare},  for integers $m \geq 4$ we have
\begin{align}\label{S22}
 \sum_{\substack{n_1,\dots,n_m \geq 1\\ n_1\cdots n_m=\square}} \lambda_f(n_1)\cdots \lambda_f(n_m)W\Big(\frac{n_1}{Y^{\varepsilon}}\Big)\cdots W\Big(\frac{n_m}{Y^{\varepsilon}}\Big)
& \ll Y^{m\varepsilon/2}(\log Y)^{\frac{m(m-3)}{2}} .
\end{align}
  It follows from $\eqref{eqS2}$--$\eqref{S22}$ that 
\begin{align}
\label{H2bound}
\mathcal{H}_2 \ll_{\varepsilon, m} X Y^{m\varepsilon/2}(\log X)^{\frac{m(m-3)}{2}}.
\end{align}

  Next, we note that, by Lemma \ref{character} and arguing similarly to our discussion of $\mathcal{H}_2$, we obtain under GRH that
\begin{align*}
\mathcal{H}_1=&\sum_{2<p\leq X}(\log p) \sum_{n\geq  1} \lambda_f(n)W\Big(\frac{n}{Y}\Big)\sum_{\substack{n_1,\ldots,n_{m-1}\geq 1 \\ nn_1\dots n_{m-1}=\square}} \chi_{8p}(nn_1\cdots n_{m-1}) \prod_{i=1}^{m-1}\lambda_f(n_i)W\Big(\frac{n_i}{Y^{\varepsilon}}\Big) \\
= &X\sum_{\substack{n,n_1,\dots,n_{m-1} \geq 1\\ nn_1\dots n_{m-1}=\square}}\lambda_f(n)\lambda_f(n_1)\cdots \lambda_f(n_{m-1})W\Big(\frac{n}{Y}\Big)W\Big(\frac{n_1}{Y^{\varepsilon}}\Big)\cdots W\Big(\frac{n_{m-1}}{Y^{\varepsilon}}\Big)\\
& +O\Bigg(X^{1/2}\log^2 X\displaystyle\sum_{\substack{n \ll Y \\ n_1,\dots,n_{m-1} \ll Y^{\varepsilon}}}|\lambda_f(n)\lambda_f(n_1)\cdots \lambda_f(n_{m-1})|\Bigg)\\
\gg & XY^{1/2+(m-1)\varepsilon/2}(\log Y)^{\frac{m(m-3)}{2}} +
O(X^{1/2+\varepsilon}Y^{1+\varepsilon}).
\end{align*}
  It follows that, under GRH,  for any $\alpha>0$ and $X^{\varepsilon} \ll Y\ll X^{1-\alpha}$,   we have 
\begin{align}
\label{H1bound}
\mathcal{H}_1 \gg XY^{1/2+(m-1)\varepsilon/2}(\log Y)^{\frac{m(m-3)}{2}}\gg XY^{1/2+(m-1)\varepsilon/2}(\log X)^{\frac{m(m-3)}{2}}.
\end{align}

 Finally, applying \eqref{Holder}, \eqref{H2bound}, and \eqref{H1bound}, we deduce under GRH that for any even integer $m \geq 4$ and any $\alpha>0$,  with $X^{\varepsilon} \ll Y\ll X^{1-\alpha}$,
   $$U_m(X,Y;f,W)  \gg XY^{m/2}(\log X)^{\frac{m(m-3)}{2}}.$$
This completes the proof of Theorem \ref{lower}.

\vspace*{.5cm}

\noindent{\bf Acknowledgments.}  P. G. is supported in part by NSFC grant 12471003.

\bibliographystyle{plain}
\bibliography{biblio}

\end{document}